\documentclass[12pt]{amsart}
\usepackage{amsmath,amsfonts,amssymb,amsthm}
\renewcommand{\emptyset}{\varnothing}
\renewcommand{\Im}{\operatorname{Im}}
\renewcommand{\Re}{\operatorname{Re}}

\renewcommand{\epsilon}{\varepsilon}
\newcommand{\SL}{\operatorname{SL}(2,{\mathbb R})}
\newcommand{\SLZ}{\operatorname{SL}(2,{\mathbb Z})}
\newcommand{\PSLZ}{\operatorname{PSL}(2,{\mathbb Z})}

\newcommand{\CP}{{\mathbb C}\!\operatorname{P}^1}
\newcommand{\lcm}{\operatorname{lcm}}
\newcommand{\MNa}{M_N(a_1,a_2,a_3,a_4)}
\newcommand{\modN}{\ (\textrm{ mod } N)}

\newcommand\Z{\mathbb Z}

\newcommand\Q{\mathbb Q}

\newcommand\R[1]{{\mathbb R}^{#1}}

\newcommand\C[1]{{\mathbb C}^{#1}}

\newcommand{\cC}{{\mathcal C}}
\newcommand{\cD}{{\mathcal D}}

\newcommand{\cH}{{\mathcal H}}
\newcommand{\cM}{{\mathcal M}}

\newlength{\halfbls}

\newtheorem{Theorem}{Theorem}

\newtheorem*{NNTheorem}{Theorem}

\newtheorem{prop}{Proposition}

\newtheorem{Proposition}[prop]{Proposition}

\newtheorem{Lemma}{Lemma}

\newtheorem{cor}{Corollary}

\newtheorem{Corollary}[cor]{Corollary}

\theoremstyle{remark}

\newtheorem*{NNConvention}{Convention}

\newtheorem{Remark}{Remark}

\newtheorem*{NNRemark}{Remark}

\theoremstyle{definition}

\newtheorem{Convention}{Convention}

\begin{document}

\begin{picture}(0,0)(0,0)
\put(-13,30){Version of February 25, 2011}
\end{picture}

\title[Lyapunov spectrum of square-tiled cyclic covers]
{Lyapunov spectrum of\\ square-tiled cyclic covers}

\author{Alex Eskin}
\thanks{Research  of  the first author is partially supported  by
NSF grant}
\address{
Department of Mathematics,
University of Chicago,
Chicago, Illinois 60637, USA\\
}
\email{eskin@math.uchicago.edu}

\author{Maxim Kontsevich}
\address{IHES,
le Bois Marie,
35, route de Chartres,
91440 Bures-sur-Yvette, FRANCE\\
}
\email{maxim@ihes.fr}

\author{Anton Zorich}
\thanks{Research of the third author is partially supported by ANR and PICS grants}
\address{
IRMAR,
Universit\'e Rennes-1,
Campus de Beaulieu,
35042 Rennes, cedex, France
}
\email{Anton.Zorich@univ-rennes1.fr}

\subjclass[2000]{
Primary
30F30, 
32G15, 
32G20, 
57M50; 
Secondary
14D07, 
37D25  
}

\keywords{Teichm\"uller geodesic flow, moduli space of quadratic
differentials, Lyapunov exponent, Hodge norm, cyclic cover}

\begin{abstract}
A  cyclic cover over $\CP$ branched at four points inherits a natural
flat  structure  from  the  ``pillow''  flat  structure  on the basic
sphere.  We  give  an  explicit  formula  for all individual Lyapunov
exponents  of  the  Hodge  bundle  over  the corresponding arithmetic
Teichm\"uller  curve.  The  key  technical  element  is evaluation of
degrees  of  line  subbundles  of  the Hodge bundle, corresponding to
eigenspaces of the induced action of deck transformations.
\end{abstract}
\maketitle

\tableofcontents


\section*{Introduction}

The current paper is an application of a long-standing project of the
authors   on   Lyapunov   exponents  for  the  Hodge  bundle  of  the
\mbox{Teichm\"uller}   geodesic  flow~\cite{Eskin:Kontsevich:Zorich}.
This   application  was  inspired  by  the  papers~\cite{ForniSurvey}
and~\cite{Forni:Matheus}, where G.~Forni and C.~Matheus observed that
special  cyclic  covers  of  $\CP$  with  four  branch points (called
square-tiled  cyclic  covers)  induce arithmetic Teichm\"uller curves
with peculiar properties.

The geometry and the topology of cyclic covers is well explored; some
ideas   applied   in   our   paper   are   very   close  to  ones  of
I.~Bouw~\cite{Bouw},  I.~Bouw and M.~M\"oller~\cite{Bouw:Moeller}, of
C.~McMullen~\cite{McMullen},    and    of    G.~Forni,    C.~Matheus,
A.~Zorich~\cite{Forni:Matheus:Zorich},  who  studied  similar  cyclic
covers    in   a   similar   context   (see   also   the   paper   of
M.~Schmoll~\cite{Schmoll:D:symmetric}  where  certain class of cyclic
covers appear under the name ``$d$-symmetric differentials'').

In  the  present  paper we give a simple explicit expression for {\it
all}  individual  Lyapunov  exponents  for  the Hodge bundle over the
ergodic components of the Teichm\"uller geodesic flow associated with
the  above-mentioned  cyclic  covers  of  $\CP$.  We  show that these
Lyapunov  exponents  have a purely geometric interpretation: they are
expressed  in  terms  of  degrees  of  line  bundles contained in the
eigenspace  decomposition  of  the  Hodge  bundle with respect to the
induced   action   of   deck  transformations.  It  was  observed  by
D.~Chen~\cite{Chen}  that  the  sum  of  these  Lyapunov exponents is
closely   related   to   the   slope   of  the  Teichm\"uller  curves
parameterizing square-tiled cyclic covers.

\bigskip

\noindent\textbf{Hodge norm.}
A  complex  structure  on  the  Riemann surface $X$ underlying a flat
surface  $S$  of genus $g$ determines a complex $g$-dimensional space
of   holomorphic   1-forms   $\Omega(X)$   on   $X$,  and  the  Hodge
decomposition
$$
H^1(X;\C{}) = H^{1,0}(X)\oplus H^{0,1}(X) \simeq
\Omega(X)\oplus\bar \Omega(X)\ .
$$
The intersection form
\begin{equation}
\label{eq:intersection:form}
\langle\omega_1,\omega_2\rangle:=\frac{i}{2}
\int_{C} \omega_1\wedge  \overline{\omega_2}\qquad\qquad\qquad
\end{equation}
is   positive-definite   on  $H^{1,0}(X)$  and  negative-definite  on
$H^{0,1}(X)$.

The    projections    $H^{1,0}(X)\to    H^1(X;\R{})$,    acting    as
$[\omega]\mapsto[\Re(\omega)]$ and $[\omega]\mapsto[\Im(\omega)]$ are
isomorphisms  of vector spaces over $\R{}$. The Hodge operator $\ast:
H^1(X;\R{})\to   H^1(X;\R{})$  acts  as  the  inverse  of  the  first
isomorphism composed with the second one. In other words, given $v\in
H^1(X;\R{})$, there exists a unique holomorphic form $\omega(v)$ such
that   $v=[\Re(\omega(v))]$;   the   dual  $\ast  v$  is  defined  as
$[\Im(\omega)]$.

Define the \textit{Hodge norm} of $v\in H^1(X,\R{})$ as
$$
\|v\|^2=\langle\omega(v),\omega(v)\rangle
$$
\smallskip

\noindent\textbf{Gauss--Manin connection.}
Passing  from  an  individual  Riemann  surface  to  the moduli stack
$\cM_g$    of    Riemann    surfaces,    we    get   vector   bundles
$H^1_{\C{}}=H^{1,0}\oplus  H^{0,1}$,  and  $H^1_{\R{}}$  over $\cM_g$
with    fibers    $H^1(X,\C{})=H^{1,0}(X)\oplus    H^{0,1}(X)$,   and
$H^1(X,\R{})$ correspondingly over $X\in\cM_g$.
\medskip

Using    the   integer   lattices   $H^1(X,\Z{}\oplus   i\Z{})$   and
$H^1(X,\Z{})$   in   the  fibers  of  these  vector  bundles  we  can
canonically  identify  fibers  over  nearby  Riemann  surfaces.  This
identification  is  called  the \textit{Gauss--Manin} connection. The
Hodge  norm \textit{is not} preserved by the Gauss---Manin connection
and  the splitting $H^1_{\C{}}=H^{1,0}\oplus H^{0,1}$ \textit{is not}
covariantly constant with respect to this connection.

The  complex  vector  bundle  $H^{1,0}$ carries a natural holomorphic
structure,  and is called the {\it Hodge} bundle. The underlying real
smooth  vector  bundle is canonically isomorphic to the cohomological
bundle   $H^1_{\R{}}$   endowed  with  flat  Gauss-Manin  connection.
Slightly  abusing  the  language,  we  shall use the same name ``Hodge
bundle''  for  the  flat  bundle $H^1_{\R{}}$. Also, we shall use the
same terminology for the pullback of $H^{1,0}\simeq H^1_{\R{}}$ under
a  holomorphic  map  from  a complex algebraic curve (which will be a
cover of the Teichm\"uller curve) to $\cM_g$.
\smallskip

\smallskip

\noindent\textbf{Lyapunov exponents.}
Informally,  the Lyapunov exponents of a vector bundle endowed with a
connection  can  be  viewed  as  logarithms  of  mean  eigenvalues of
monodromy of the vector bundle along a flow on the base.

In  the  case  of  the  Hodge  bundle considered as the cohomological
bundle,  we  take  a  fiber  of  $H^1_{\R{}}$  and  pull  it  along a
Teichm\"uller geodesic on the moduli space. We wait till the geodesic
winds a lot and comes close to the initial point and then compute the
resulting  monodromy matrix $A(t)$. Finally, we compute logarithms of
eigenvalues  of  $A^T\!A$, and normalize them by twice the length $t$
of the geodesic. By the Oseledets multiplicative ergodic theorem, for
almost   all  choices  of  initial  data  (starting  point,  starting
direction)  the  resulting  $2g$  real  numbers  converge  as  $t \to
\infty$,  to  limits  which  do not depend on the initial data. These
limits     $\lambda_1\ge\dots\ge\lambda_{2g}$    are    called    the
\textit{Lyapunov   exponents}   of   the   Hodge   bundle  along  the
Teichm\"uller flow.

The  matrix  $A(t)$ preserves the intersection form on cohomology, so
it  is  symplectic.  This implies that Lyapunov spectrum of the Hodge
bundle is symmetric with respect to a sign interchange,
$
\lambda_j=-\lambda_{2g-j+1}
$.
Moreover,
{
if  the  base of the bundle is located in the moduli space $\cH_g$ of
\textit{holomorphic $1$-forms},
}
from  elementary  geometric  arguments it follows that one always has
$\lambda_1=1$.   Thus,  the  Lyapunov  spectrum  is  defined  by  the
remaining nonnegative Lyapunov exponents
$$
\lambda_2\ge\dots\ge\lambda_g\, .
$$

\begin{Convention}
\label{conv:nonnegative:part}
The      collection      of     the     leading     $g$     exponents
$\{\lambda_1,\dots,\lambda_g\}$      will      be      called     the
\textit{nonnegative  part}  of  the  Lyapunov  spectrum  of the Hodge
bundle  $H^1_{\R{}}$.  We  warn  the  reader,  that  when some of the
exponents  are  null,  the  ``\textit{nonnegative  part}''
contains only  half of
\textit{all} zero exponents.
\end{Convention}

\smallskip

\noindent\textbf{Cyclic covers.}
Consider   an   integer   $N\ge   1$  and  a  $4$-tuple  of  integers
$(a_1,\dots,a_4)$ satisfying the following conditions:
\begin{equation}
\label{eq:a1:a4}
0<a_i\le N\,;\quad \gcd(N, a_1,\dots,a_4) =1\,;\quad
\sum\limits_{i=1}^4 a_i\equiv 0 \modN\ .
\end{equation}
Let  $z_1,z_2,z_3,z_4\in  \C{}$ be four distinct points. By $\MNa$ we
denote the closed connected nonsingular Riemann surface obtained from
the one defined by the equation
\begin{equation}
\label{eq:cyclic:cover}
w^N=(z-z_1)^{a_1}(z-z_2)^{a_2}(z-z_3)^{a_3}(z-z_4)^{a_4}
\end{equation}
by   normalization.  By  construction,  $M_N(a_1,a_2,a_3,a_4)$  is  a
ramified cover over the Riemann sphere $\CP$ branched over the points
$z_1, \dots, z_4$. The group of deck transformations of this cover is
the cyclic group $\Z/N\Z$ with a generator $T:M\to M$ given by
\begin{equation}
\label{eq:T}
T(z,w)=(z,\zeta w) \,,
\end{equation}
where   $\zeta$   is   a   primitive   $N$th   root  of  unity  (e.g.
$\zeta=\zeta_N=\exp(2\pi i/N)$).

One  can  also  consider the case when one of the ramification points
$z_i$  is located at infinity. In this case one just skips the factor
$(z-z_i)^{a_i}$  in~\eqref{eq:cyclic:cover}.  Notice that the integer
parameter  $a_i$  is  uniquely  determined by three remaining numbers
$(a_j)_{j\ne i}$ by relations~\eqref{eq:a1:a4}.

Throughout  this  paper  by a \textit{cyclic cover} we call a Riemann
surface   $M_N(a_1,\dots,a_4)$,   with  parameters  $N,a_1,\dots,a_4$
satisfying relations~\eqref{eq:a1:a4}.
\smallskip

\noindent\textbf{Square-tiled surface associated to a cyclic cover.}
Any  meromorphic  quadratic  differential  $q(z)(dz)^2$  with at most
simple poles on a Riemann surface defines a flat metric $g(z)=|q(z)|$
with  conical  singularities  at  zeroes and poles of $q$. Consider a
meromorphic quadratic differential
\begin{equation}
\label{eq:q:on:CP1}
q_0=\frac{
  {c_0}
(dz)^2}{(z-z_1)(z-z_2)(z-z_3)(z-z_4)}\,, \quad c_0\in \C{}\setminus\{0\}\,,
\end{equation}
on  $\CP$.  It  has  simple  poles  at $z_1,z_2,z_3,z_4$ and no other
zeroes  or  poles.  The  quadratic  differential $q_0$ defines a flat
metric  on  a  sphere  obtained  by  identifying  two  copies  of  an
appropriate      parallelogram     by     their     boundary,     see
Figure~\ref{fig:flat:sphere}.

\begin{figure}[htb]
\includegraphics{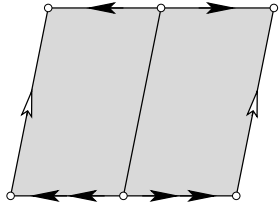}
\vspace{80pt}
\caption{
\label{fig:flat:sphere}
Flat sphere with four conical singularities having cone angles $\pi$.
}
\end{figure}

For a convenient choice of parameters $z_1,\dots,z_4$
{and $c_0$}
the  parallelogram  becomes  the  unit  square.  Metrically, we get a
square  pillow  with  four corners corresponding to the four poles of
$q_0$.

Now  consider  some  cyclic cover $\MNa$ and the canonical projection
$p:\MNa\to\CP$.  Consider an induced quadratic differential $q=p^\ast
q_0$  on  $\MNa$  and the corresponding flat metric. By construction,
the  resulting  flat  surface  is  tiled  with unit squares. In other
words, we get a \textit{square-tiled surface}.

Throughout  this  paper we work only with such square-tiled (or, more
generally,  parallelogram-tiled)  cyclic  covers. We refer the reader
to~\cite{Forni:Matheus:Zorich}  for  a  description  of the geometry,
topology and combinatorics of square-tiled cyclic covers.

\begin{Convention}
\label{conv:naming:ramification:points}
It  would  be convenient to ``give names'' to the ramification points
$z_1,z_2,z_3,z_4$.  In  other  words, throughout this paper we assume
that $z_i, z_j$ are distinguishable even if $a_i=a_j$.
\end{Convention}

Varying  the  cross-ratio  $(z_1:z_2:z_3:z_4)$  (e.g.  keeping 3 of 4
points fixed and varying the fourth point) we obtain the moduli curve
which we denote by
$$
\mathcal{M}_{(a_i),N}\,\,.
$$
As   an   abstract   curve   it  is  isomorphic  to  $\cM_{0,4}\simeq
\C{}\!P^1-\{0,1,\infty\}$. Strictly speaking, it should be considered
as a stack, because every curve $\MNa$ with named ramification points
has  an  automorphism  group  canonically isomorphic to $\Z/N\Z$. The
naive notion of an orbifold is not sufficient here because the moduli
stack  in  our situation is isomorphic to the quotient of $\cM_{0,4}$
by  the  {\it  trivial}  action of $\Z/N\Z$. In what follows we shall
treat  $\mathcal{M}_{(a_i),N}$  as  a  plain  curve,  in  order to be
elementary.

The   curve   $\mathcal{M}_{(a_i),N}$   maps  onto  the  image  of  a
Teichm\"uller  disc in the moduli stack of Abelian differentials with
certain  multiplicities of zeroes, and forgetting the differential we
obtain     a     Tecihm\"uller     geodesic    in    $\cM_g$    where
$g=g(N,a_1,\dots,a_4)$  is  the genus of curve $M_N(a_1,a_2,a_3,a_4)$
(see   Section   2.1   for   an  explicit  formula  for  $g$).  Hence
$\mathcal{M}_{(a_i),N}$  is  a Teichm\"uller curve (e.g. in the sense
of \cite{Bouw:Moeller}). This is the main object of our study.

\section{Statement of results}

\subsection{Splitting of the Hodge bundle}
Consider a cyclic cover
$$
X=\MNa
$$
over  $\CP$ defined by equation~\eqref{eq:cyclic:cover}. Consider the
canonical generator $T$ of the group of deck transformations; let
$$
T^\ast: H^1(X;\C{})\to H^1(X;\C{})
$$
be      the      induced      action     in     cohomology.     Since
$(T^\ast)^N=\operatorname{Id}$, the eigenvalues of $T^\ast$ belong to
a       subset      of      $\{\zeta,\dots,\zeta^{N-1}\}$,      where
$\zeta=\exp\left(\dfrac{2\pi  i}{N}\right)$.  We  excluded  the  root
$\zeta^0=1$   since   any   cohomology  class  invariant  under  deck
transformations  would  be a pullback of a cohomology class on $\CP$,
and $H^1(\CP)=0$.

For $k=1,\dots,N-1$ denote
$$
V(k)(X):=\operatorname{Ker}(T^\ast-\zeta^k\operatorname{Id})
\subset H^1(X;\C{})\ .
$$
The decomposition
$$
H^1(X;\C{})=\oplus V(k)(X)\,,
$$
is  preserved  by  the Gauss-Manin connection, which implies that the
vector bundle $H^1_{\C{}}$ over the Teichm\"uller curve splits into a
sum of invariant subbundles $V(k)$.

Denote
$$
V^{1,0}(k):=V(k)\cap H^{1,0}\quad
\text{ and }\quad
V^{0,1}(k):=V(k)\cap H^{0,1}\ .
$$
Since  a  generator $T$ of the group of deck transformations respects
the complex structure, it induces a linear map
$$
T^\ast: H^{1,0}(X)\to H^{1,0}(X)\,.
$$
This map preserves the Hermitian form~\eqref{eq:intersection:form} on
$H^{1,0}(X)$.  This  implies  that $T^\ast$ is a unitary operator on
$H^{1,0}(X)$, and hence $H^{1,0}(X)$ admits a splitting into a direct
sum of eigenspaces of $T^\ast$,
\begin{equation}
\label{eq:H10:direct:sum:for:k}
H^{1,0}(X)=\bigoplus_{k=1}^{N-1}V^{1,0}(k)(X)\ .
\end{equation}
The  latter observation also implies that for any $k=1,\dots,N-1$ one
has    $V(k)=V^{1,0}(k)\oplus    V^{0,1}(k)$.   The   vector   bundle
$V^{1,0}(k)$  over the Teichm\"uller curve is a holomorphic subbundle
of $H^1_{\C{}}$.

Denote
\begin{equation}
\label{eq:tik}
t_i(k)=\left\{\frac{a_i}{N}k\right\}\ ,k=1,\dots,N-1\ ,
\end{equation}
where $\{x\}$ denotes the fractional part of $x$. Let
\begin{equation}
\label{eq:tk}
t(k)=t_1(k)+\dots+t_4(k)\ .
\end{equation}
Since  $(a_1+\dots+a_4)/N\in\{1,2,3\}$ and $\gcd(N,a_1,\dots,a_4)=1$,
we get\ $t(k)\in\{1,2,3\}$.

\begin{NNTheorem}[I.~Bouw]
For any $k=1,\dots,N-1$ one has
\begin{align}
\label{eq:dim:duality}
&\dim_{\C{}} V^{1,0}(k)=\dim_{\C{}} V^{0,1}(N-k)= t(N-k)-1\\
\label{eq:dim:Vk}
&\dim_{\C{}} V(k)=t(k)+t(N-k)-2\in\{0,1,2\}
\end{align}
\end{NNTheorem}


For  the  sake  of completeness we provide a proof of this Theorem in
section~\ref{ss:basis}.


It would be convenient to state the following elementary observation.
\begin{Lemma}
\label{lm:dim:Vk:equals:two}
The eigenspace $V(k)$ has complex dimension
\begin{itemize}
\item
two, if and only if $t_i(k)>0$ for $i=1,2,3,4$;
\item one, if and only if there is exactly one $i\in\{1,2,3,4\}$ such
that $t_i(k)=0$;
\item
zero, if there are two distinct indices $i,j\in\{1,2,3,4\}$ such that
$t_i(k)=t_j(k)=0$.
\end{itemize}
\end{Lemma}
\begin{proof}
By formula~\eqref{eq:dim:Vk} of I.~Bouw one has
\begin{multline*}
\dim_{\C{}}V(k)+2\ =\ t(k)+t(N-k)
\ =\\=\
\left(\left\{\frac{a_1}{N}k\right\}+
\left\{\frac{a_1}{N}(N-k)\right\}
\right)+
\dots+
\left(\left\{\frac{a_4}{N}k\right\}+
\left\{\frac{a_4}{N}(N-k)\right\}
\right)
\end{multline*}
see~\eqref{eq:tik}  and~\eqref{eq:tk}. The statement of the Lemma now
follows  from  the  following elementary remark. If $x+y$ is integer,
then
$$
\{x\}+\{y\}=
\begin{cases}
1&\text{when }\{x\}>0,\\
0&\text{otherwise .}
\end{cases}
$$
\end{proof}

The  Teichm\"uller  curve  is  not  compact, but there is a canonical
extension  of  holomorphic  bundles  $V^{1,0}(k)$  to orbifold vector
bundles  at  the  cusps  (see  the next section). Hence, we can speak
about the (orbifold) degree of these bundles.
\begin{Theorem}
\label{th:degree}  For  any  $k$  such  that  $t(N-k)=2$ the orbifold
degree of the line bundle $V^{1,0}(k)$ satisfies:
\begin{equation}
\label{eq:dk:equals:min}
d(k)=\min\big(t_1(k),1-t_1(k),\dots,t_4(k),1-t_4(k)\big) \ .
\end{equation}
Moreover, if $t(N-k)=2$ the following alternative holds: if $t(k)=2$,
then $d(k)>0$; if $t(k)=1$, then $d(k)=0$.
\end{Theorem}

\begin{Remark}
\label{rm:dim:2:implies:dgree:0}
Note that $t(N-k)\in\{1,2,3\}$.
{Theorem~\ref{th:degree} studies the case, when
$t(N-k)=2$.}
Thus,  there  remain  two  complementary  cases  when $t(N-k)\neq 2$.
Namely,  when  $t(N-k)=1$  we get $\dim V^{1,0}(k)=0$, and the bundle
$V^{1,0}(k)$  is missing. When $t(N-k)=3$ we get $\dim V^{1,0}(k)=2$.
We  shall  show that in this case the orbifold degree of $V^{1,0}(k)$
is equal to zero.
\end{Remark}


\subsection{On   real  and  complex  variations  of  polarized  Hodge
structures of weight $1$}

Let  $\cC$  be a smooth possibly non-compact complex algebraic curve.
We  recall  that a variation of {\it real} polarized Hodge structures
of  weight $1$  on $\cC$ is given by a real symplectic vector bundle
$\mathcal{E}_{\R{}}$  with  a flat connection $\nabla$ preserving the
symplectic  form,  such  that  every  fiber of $\mathcal E$ carries a
Hermitian  structure  compatible  with  the symplectic form, and such
that     the     corresponding     complex    Lagrangian    subbundle
$\mathcal{E}^{1,0}$   of  the  complexification  $\mathcal{E}_{\C{}}=
\mathcal{E}_{\R{}}\otimes\C{}$ is {\it holomorphic}. The variation is
called  {\it  tame}  if all eigenvalues of the monodromy around cusps
lie  on  the  unit  circle,  and the subbundle $\mathcal{E}^{1,0}$ is
meromorphic  at cusps. For example, the Hodge bundle of any algebraic
family  of  smooth compact curves over $\cC$ (or an orthogonal direct
summand of it) is a tame variation.

Similarly, a variation of {\it complex} polarized Hodge structures of
weight $1$  is given by a complex vector bundle $\mathcal{E}_{\C{}}$
of  rank  $r+s$ (where $r,s$ are nonnegative integers) endowed with a
flat  connection $\nabla$, by a covariantly constant pseudo-Hermitian
form   of   signature   $(r,s)$,   and  by  a  holomorphic  subbundle
$\mathcal{E}^{1,0}$  of  rank  $r$,  such that the restriction of the
form  to  it  is  strictly  positive.  The  condition  of tameness is
completely parallel to the real case.

Any  real  variation  of  rank  $2r$ gives a complex one of signature
$(r,r)$  by  the complexification. Conversely, one can associate with
any        complex       variation       $(\mathcal{E}_{\C{}},\nabla,
\mathcal{E}^{1,0})$  of  signature  $(r,s)$  a real variation of rank
\mbox{$2(r+s)$},  whose  underlying  local  system of real symplectic
vector spaces is obtained from $\mathcal{E}_{\C{}}$ by forgetting the
complex structure.

Let   us  assume  that  the  variation  of  complex  polarized  Hodge
structures of weight $1$ has a unipotent monodromy around cusps. Then
the   bundle   $\mathcal{E}^{1,0}$   admits   a  canonical  extension
$\overline{\mathcal{E}^{1,0}}$   to   the   natural  compactification
$\overline{\mathcal{C}}$.  It  can  be described as follows: consider
first     an     extension     $\overline{\mathcal{E}_{\C{}}}$     of
$\mathcal{E}_{\C{}}$    to   $\overline{\mathcal{C}}$   as   a   {\it
holomorphic} vector bundle in such a way that the connection $\nabla$
will  have  only first order poles at cusps, and the residue operator
at  any  cup is nilpotent (it is called the {\it Deligne extension}).
Then     the    holomorphic    subbundle    $\mathcal{E}^{1,0}\subset
\mathcal{E}_\C{}$     extends     uniquely     as     a     subbundle
$\overline{\mathcal{E}^{1,0}}\subset  \overline{\mathcal{E}_\C{}}$ to
the cusps.

\subsection{Sum of Lyapunov exponents of an invariant subbundle}
\label{ss:theorem:sum}

Let   $(\mathcal{E}_{\R{}},\nabla,\mathcal{E}^{1,0})$   be   a   tame
variation  of polarized real Hodge structures of rank $2r$ on a curve
$\cC$  with  {\it  negative} Euler characteristic. For example, $\cC$
could  be  an  unramified cover of a general arithmetic Teichm\"uller
curve,  and  $\mathcal  E$  could  be a subbundle of the Hodge bundle
which  is  simultaneously invariant under the Hodge star operator and
under the monodromy.

Using  the  canonical  complete  hyperbolic  metric  on $\cC$ one can
define  the  geodesic  flow  on  $\cC$ and the corresponding Lyapunov
exponents  $\lambda_1\ge  \dots \ge \lambda_{2r}$ for the flat bundle
$(\mathcal{E}_{\R{}},\nabla)$, satisfying the usual symmetry property
$\lambda_{2r+1-i}=-\lambda_i,\,i=1,\dots, r$.

The holomorphic vector bundle $\mathcal{E}^{1,0}$ carries a Hermitian
form, hence its top exterior power $\wedge^r(\mathcal{E}^{1,0})$ is a
holomorphic  line bundle also endowed with a Hermitian metric. Let us
denote  by $\alpha$ the curvature $(1,1)$-form on $\cC$ corresponding
to  this metric, divided by $-(2\pi i)$. The form $\alpha$ represents
the first Chern class of $\mathcal{E}^{1,0}$.

Then we have the following general result:
\begin{NNTheorem}
Under  the  above  assumptions,  the  sum  of  the  top  $r$ Lyapunov
exponents of $V$ with respect to the geodesic flow satisfies
\begin{equation}
\label{eq:sum:lambda}
\lambda_1+\dots+\lambda_r=
\cfrac{2\int_C \alpha}{2g_\cC-2+c_\cC}\,,
\end{equation}
where  we  denote  by  $g_\cC$ --- the genus of $\cC$, and by $c_\cC$
---    the    number    of    hyperbolic    cusps    on   $\cC$.
\end{NNTheorem}

This  formula  was formulated (in a slightly different form) first in
~\cite{Kontsevich}      and     then     proved     rigorously     by
G.~Forni~\cite{Forni}.

Note  that  a  similar  result  holds  also for \textit{complex} tame
variations  of polarized Hodge structures. Namely, for a variation of
signature  $(r,s)$ one has $r+s$ Lyapunov exponents with sum equal to
$0$:
$$
\lambda_1\ge \dots\ge \lambda_r\ge 0\ge
        \lambda_{r+1}\ge\dots\ge \lambda_{r+s}\,.
$$
The  collection (with multiplicities) $\{\lambda_1,\dots,\lambda_r\}$
will  be  called  the  \textit{non-negative}  part  of  the  Lyapunov
spectrum.   We   claim   that   the  sum  of  non-negative  exponents
$\lambda_1+\dots+\lambda_r$     is     again     given     by     the
formula~\eqref{eq:sum:lambda}.

The  proof follows from the simple observation that one can pass from
a complex variation to a real one by taking the underlying real local
system.  Both  the  sum of non-negative exponents and the integral of
the Chern form are multiplied by two under this procedure.

The  denominator  in  the  above  formula is equal to minus the Euler
characteristic  of $\cC$, i.e. to the area of $\cC$ up to a universal
factor   $2\pi$.  The  numerator  also  admits  an  algebro-geometric
interpretation  for  variations  of  real Hodge structures arising as
direct  summands  of  Hodge bundles for algebraic families of curves.
Namely,  let  us  assume that the monodromy of $(\mathcal{E},\nabla)$
around  any  cusp  is unipotent (this can be achieved by passing to a
finite  unramified  cover  of  $\cC$).  Then  one  has  the following
identity (see e.g. Proposition 3.4 in~\cite{Peters}):
$$
\int_{\mathcal C} \alpha= \deg\overline{\mathcal{E}^{1,0}}\,.
$$
In  general, without the assumption on unipotency, we obtain that the
integral  (in  the  numerator)  is  a  rational  number, which can be
interpreted  as  an  orbifold  degree  in  the following way. Namely,
consider  an  unramified  Galois  cover $\mathcal{C}'\to \mathcal{C}$
such  that  the  pullback  of  $(\mathcal{E},\nabla)$ has a unipotent
monodromy.  Then the compactified curve $\overline{\mathcal{C}}$ is a
quotient  of  $\overline{\mathcal{C}'}$ by a finite group action, and
hence  is  endowed  with  a natural orbifold structure. Moreover, the
holomorphic Hodge bundle on $\overline{\mathcal{C}'}$ will descend to
an  orbifold bundle on $\overline{\mathcal{C}}$. Then the integral of
$\alpha$  over  $\mathcal  C$ is equal to the orbifold degree of this
bundle.

The  choice  of the orbifold structure on $\overline{\mathcal{C}}$ is
in  a  sense  arbitrary,  as we can choose the cover $\mathcal{C}'\to
\mathcal{C}$  in  different  ways. The resulting orbifold degree does
not  depend  on  this  choice.  The  corresponding  algebro-geometric
formula  for  the  denominator given as an orbifold degree, is due to
I.~Bouw  and  M.~M\"oller  in~\cite{Bouw:Moeller}.  In  our  concrete
example          of          the          Teichm\"uller         curve
$\mathcal{C}=\mathcal{M}_{(a_i),N}\simeq
\C{}P^1\setminus\{0,1,\infty\}$  an explicit convenient choice of the
cover  is  the  standard  Fermat  curve  $\mathcal{C}'=F_N$  given by
equation  $x^N+y^N=1,\,\,x,y\in{\C{}}^*$,  with  the  projection  map
$(x,y)\mapsto x^N\in \C{}\setminus\{0,1\}$.

\subsection{Lyapunov spectrum for cyclic covers}
\label{ss:Lyapunov:spectrum:for:cyclic:covers}

Recall  that  we have a decomposition of the holomorphic Hodge bundle
over $\mathcal{C}=\mathcal{M}_{(a_i),N}$ into a direct sum
$$
H^{1,0}=\bigoplus_{1\le k\le N-1}V^{1,0}(k)\,,
$$
coming   from  the  decomposition  of  {\it  complex}  variations  of
polarized  Hodge  structures.  It  induces  a  decomposition  of  the
variation of {\it real} polarized Hodge structures
$$
H^1_{\R{}}=\bigoplus_{1\le k\le N/2} W_{\R{}}(k)
$$
in a way described below.

First  consider the case where $k$ is an integer such that $1\le k\le
N-1,\,k\ne N/2$. By $W(k)\subset H^1_{\R{}}$ denote the projection of
the  subbundle  $V(k)\oplus  V(N-k)\subset  H^1_{\C{}}$  to the first
summand in the canonical decomposition $H^1_{\C{}}=H^1_{\R{}}\oplus i
H^1_{\R{}}$.

By  definition  $W(N-k)=W(k)$. Note that the roots of unity $\zeta^k$
and   $\zeta^{N-k}$   are  complex  conjugate.  Hence,  the  subspace
$V(k)\oplus  V(N-k)$  is  invariant  under  complex conjugation, and,
thus,
$$
W_{\C{}}(k)=W_{\C{}}(N-k)=V(k)\oplus V(N-k)\ .
$$

Since    the    bundle    $H^{1,0}$    decomposes   into   a   direct
sum~\eqref{eq:H10:direct:sum:for:k} of eigenspaces, we conclude that
$$
W_{\C{}}(k)=W^{1,0}(k)\oplus W^{0,1}(k)\ ,
$$
where
\begin{align*}
W^{1,0}(k)&=V^{1,0}(k)\oplus V^{1,0}(N-k)\\
W^{0,1}(k)&=V^{0,1}(k)\oplus V^{0,1}(N-k)\ .
\end{align*}
The   subbundle  $W(k)$  of  $H^1_{\R{}}$  is  Hodge  star-invariant.

Note  that  the  subbundles  $V(k),  V(N-k)\subset H^1_{\C{}}$ of the
Hodge      bundle      and      the      canonical      decomposition
$H^1_{\C{}}=H^1_{\R{}}\oplus  iH^1_{\R{}}$  are  covariantly constant
with  respect  to  the  Gauss--Manin connection. Hence, the subbundle
$W(k)$  is also covariantly constant with respect to the Gauss--Manin
connection.

When  $N$  is even, denote by $W(N/2)$ the projection of the
subbundle
$V(N/2)\subset  H^1_{\C{}}$  to  the  first summand in the canonical
decomposition  $H^1_{\C{}}=H^1_{\R{}}\oplus i H^1_{\R{}}$. Similar
to the previous case,
$$
W_{\C{}}(N/2)=V(N/2)=V^{1,0}(N/2)\oplus V^{0,1}(N/2)\ ;
$$
$W(N/2)\subset  H^1_{\R{}}$  is  Hodge star-invariant and covariantly
constant with respect to the Gauss--Manin connection.

The   decomposition   of  $H^1_{\C{}}$  into  a  direct  sum  of  the
eigenspaces  $V(k)$  implies  a  decomposition  into  a direct sum of
variations of real polarized Hodge structures of weight $1$
$$
H^1_{\R{}}=\bigoplus_{1\le k\le N/2} W(k)\,.
$$

Now everything is ready to formulate the Main Theorem. Recall that by
definition~\eqref{eq:tk}   of   $t(k)$  one  has  $t(k),  t(N-k)  \in
\{1,2,3\}$. By formula~\eqref{eq:dim:Vk} of I.~Bouw one has
$$
\dim_\C{} V(k)+2=t(k)+t(N-k)\in\{2,3,4\}\,.
$$
Thus,  the  Theorem  below describes all possible combinations of the
values of $t(k)$ and $t(N-k)$.

\begin{Theorem}
\label{th:main:theorem}
For any integer $k$ such that $1\le k< N/2$ the Lyapunov exponents of
the  invariant subbundle $W(k)$ of the Hodge bundle $H^1_{\R{}}$ over
$\cC$  with  respect  to  the geodesic flow on $\cC$ are described as
follows.
\noindent $\bullet$
If $t(k)=3$, then $t(N-k)=1$, and
$$
\dim_{\C{}} V(k)=\dim_{\C{}}V^{0,1}(k)=2\,.
$$
If $t(N-k)=3$, then $t(k)=1$, and
$$
\dim_{\C{}} V(k)=\dim_{\C{}}V^{1,0}(k)=2\,.
$$
In both cases $\dim_{\R{}} W(k)=4$ and all four Lyapunov exponents of
the vector bundle $W(k)$ are equal to zero.

\noindent $\bullet$
If $t(k)=2$ and $t(N-k)=1$, then
$$
\dim_{\C{}} V(k)=\dim_{\C{}}V^{0,1}(k)=1\,.
$$
If
$t(N-k)=2$ and $t(k)=1$, then
$$
\dim_{\C{}} V(k)=\dim_{\C{}}V^{1,0}(k)=1\,.
$$
In both cases $\dim_{\R{}} W(k)=2$ and both Lyapunov exponents of the
vector bundle $W(k)$ are equal to zero.

\noindent $\bullet$
If $t(N-k)=t(k)=2$, then
$$
\dim_{\C{}} V^{1,0}(k)=\dim_{\C{}}V^{0,1}(k)=1\,,
$$
and  $\dim_{\R{}}  W(k)=4$. In this case the Lyapunov spectrum of the
vector  bundle  $W(k)$  is  equal to $\{2d(k),2d(k),-2d(k),-2d(k)\}$,
where
$$
d(k)=\min\big(t_1(k),1-t_1(k),\dots,t_4(k),1-t_4(k)\big) > 0
$$
is the orbifold degree of the line bundle $V^{1,0}(k)$.

\noindent $\bullet$
If $t(N-k)=t(k)=1$, then $V(k), V(N-k)$, and $W(k)$ vanish.

\noindent  $\bullet$
Finally,  if  $N$  is  even, and all $a_i$ are odd, then $\dim_{\R{}}
W(N/2)=2$.  In  this  case the Lyapunov spectrum of the vector bundle
$W(N/2)$ is equal to $\{1,-1\}$.

\noindent $\bullet$ If $N$ is even, but at least one of $a_i$ is also
even, then $W(N/2)$ vanishes.
\end{Theorem}
\begin{proof}
The  Lyapunov exponents of the real flat bundle $H^1_{\R{}}$ coincide
with  the  Lyapunov  exponents  of  its complexification $H^1_{\C{}}$
considered  as  a complex flat bundle (as eigenvalues and their norms
for  a  real matrix coincide with those of its complexification). The
flat  bundle  $H^1_{\C{}}$  is  decomposed  into a direct sum of flat
subbundles  $V(k),\,k=1,\dots  N-1$, underlying complex variations of
polarized  Hodge  structures.  Hence,  the  non-negative  part of the
Lyapunov  spectrum  for  $H^1_{\R{}}$  is  the  sum  over  $k$ of the
non-negative  parts of the spectra of the individual summands $V(k)$.
For  any  given  $k$  we  have  six  possibilities  for the signature
$(r,s)=\big(t(N-k)-1\,,\,t(k)-1\big)$ of the variation $V(k)$:
\begin{equation}
\label{eq:6cases}
(0,0),(1,0),(0,1),(2,0),(1,1),(0,2)
\end{equation}
according  to the theorem of I.~Bouw. In all cases except the case of
signature  $(1,1)$  the  corresponding local system is {\it unitary},
hence  all  Lyapunov  exponents  are zero. In the non-trivial case of
signature $(1,1)$ the unique non-negative Lyapunov exponent coincides
obviously  with  the sum of non-negative Lyapunov exponents and hence
can  be  calculated by formula~\eqref{eq:sum:lambda}. The denominator
in  this  formula  is  equal  to  $1$ because our Teichm\"uller curve
$\mathcal{C}=\mathcal{M}_{(a_i),N}$  is a sphere with 3 punctures and
has  Euler  characteristics $-1$. The numerator is twice the orbifold
degree  of  the determinant of the holomorphic bundle $V^{(1,0)}(k)$.
The  latter orbifold degree is the same as the orbifold degree of the
\textit{line}  bundle $V^{(1,0)}(k)$ (because $V^{(1,0)}(k)$ has rank
one),  and  is  equal to $d(k)$ by Theorem~\ref{th:degree}. Hence, we
conclude   that  the  contribution  of  the  summand  $V(k)$  to  the
non-negative part of the Lyapunov spectrum for $H^1_{\R{}}$ for cases
listed in~\eqref{eq:6cases} is given by
$$
\emptyset,\ \{0\},\ \emptyset,\ \{0,0\},\ \{2 d(k)\},\ \emptyset\,.
$$
Finally,  we  can find the non-negative Lyapunov spectrum of the real
variations  $W(k),\,1\le  k\le N/2$, using the fact that $W(k)\otimes
\C{}=V(k)\oplus  V(N-k)$ for $k< N/2$ and $W(N/2)\otimes \C{}=V(N/2)$
for  even $N$. Also, notice that in the non-trivial case of signature
$(1,1)$  one  has  $d(k)=d(N-k)>0$.  The  whole Lyapunov spectrum for
$W(k)$  is  obtained  from  the  non-negative  part  by adding a copy
reflected at zero.

\end{proof}

To    illustrate    Theorem~\ref{th:main:theorem}   we   present   in
Appendix~\ref{a:table}  a  table  of  quantities discussed above in a
particular case of $M_{30}(3,5,9,13)$.

\begin{Corollary}
\label{cr:alorithm}
The  nonnegative part $\{\lambda_1,\dots,\lambda_g\}$ of the Lyapunov
spectrum   of  the  Hodge  bundle  $H^1_{\R{}}$  over  an  arithmetic
Teichm\"uller   curve   corresponding   to   a   square-tiled  cyclic
cover~\eqref{eq:cyclic:cover}   can  be  obtained  by  the  following
algorithm.        Start        with        the        empty       set
$\operatorname{{\Lambda}Spec}=\emptyset$.          For          every
$k\in\{1,\dots,N-1\}$ compute $t(k)$ and proceed as follows:
\begin{itemize}
\item
if $t(k)=3$, then add a pair of zeroes to $\operatorname{{\Lambda}Spec}$;
\item
If $t(k)=2$, then add a number $2d(k)$ to $\operatorname{{\Lambda}Spec}$;
\item
If $t(k)=1$, then do not change $\operatorname{{\Lambda}Spec}$.
\end{itemize}
The  resulting unordered set $\operatorname{{\Lambda}Spec}$ coincides
with the unordered set $\{\lambda_1,\dots,\lambda_g\}$.
\end{Corollary}
Here  the ``nonnegative part'' of the Lyapunov spectrum is understood
in the sense of Convention~\ref{conv:nonnegative:part}.

The proof of the Corollary below is absolutely elementary, so we omit
it.

\begin{Corollary}
When  $N$  is  even  and all $a_i$ are odd, the top Lyapunov exponent
$\lambda_1=1$  is simple, $1=\lambda_1>\lambda_2$. All other strictly
positive  Lyapunov exponents of the Hodge bundle $H^1_{\R{}}$ over an
arithmetic Teichm\"uller curve corresponding to a square-tiled cyclic
cover~\eqref{eq:cyclic:cover}  are  strictly  less  than $1$ and have
even multiplicity.
\end{Corollary}

\section{Hodge structure of a cyclic cover}
\label{s:Hodge:structure}

In   section~\ref{ss:basis}   we   construct  an  explicit  basis  of
holomorphic  forms  on  any given cyclic cover $\MNa$. All forms from
this  basis  are  eigenforms under the induced action of the group of
deck  transformations,  which gives the description of the dimensions
of  $V^{1,0}(k)$  and  $V^{0,1}(k)$.  An  analogous  calculation  was
already   performed  by  I.~Bouw  and  M.~M\"oller,  see~\cite{Bouw},
\cite{Bouw:Moeller},   and   by  C.~McMullen~\cite{McMullen},  so  we
present it here mostly for the sake of completeness.

In  section~\ref{ss:degree} we compute the degree of $V^{1,0}(k)$ for
those $k$ for which $V^{1,0}(k)$ is a line bundle.

\subsection{Basis of holomorphic forms}
\label{ss:basis}

Recall that $t(k)$ was introduced in equation~\eqref{eq:tk}, and that
$\{x\}$  and  $[x]$  denote  fractional  part and integer part of $x$
correspondingly.

\begin{Lemma}
\label{lm:basis}
Consider the meromorphic form
\begin{equation}
\label{eq:general:form}
\omega=(z-z_1)^{b_1}\dots(z-z_4)^{b_4}\ \cfrac{dz}{w^k}\
\end{equation}
on  a  cyclic  cover  $\MNa$.  Fix $k$ and let the integer parameters
$b_1,\dots,b_4$ vary.
\smallskip

$\bullet$  When  $t(k)=1$,  this  form  is  not  holomorphic  for any
parameters $b_i$.

$\bullet$     When     $t(k)=2$,     the     form    $\omega_k$    as
in~\eqref{eq:general:form} is holomorphic for
$
b_i(k)=\left[\cfrac{a_i}{N}\,k\right]
$
and is not holomorphic for other values of $b_i$.

$\bullet$   When  $t(k)=3$  there  is  a  two-dimensional  family  of
holomorphic forms spanned by the forms
$
\omega_{k,1}:=(z-z_1)^{b_1}\dots(z-z_4)^{b_4}\cfrac{dz}{w^k}\,,
$
and
$
\omega_{k,2}:=z(z-z_1)^{b_1}\dots(z-z_4)^{b_4}\cfrac{dz}{w^k}\,,
$
where  $b_i(k)=\left[\cfrac{a_i}{N}\,k\right]$.  Any holomorphic form
$\omega$ as in~\eqref{eq:general:form} belongs to this family.
\end{Lemma}
\begin{proof}
To  prove  the Lemma we have to study the behavior of the meromorphic
form
$$
\omega=(z-z_1)^{b_1}\dots(z-z_4)^{b_4}\ \cfrac{dz}{w^k}\
$$
in  a  neighborhood  of  $z=z_1,\dots,z_4$  and  in a neighborhood of
$z=\infty$.

Let  $\ell=\lcm(N,a_1)$.  Consider a coordinate $u$ in a neighborhood
of a point $z_1$, such that
$$
(z-z_1)\sim u^\frac{\ell}{a_1}
$$
Then, in a neighborhood of $z_1$ we have
$$
w\sim (z-z_1)^\frac{a_1}{N}\sim u^\frac{\ell}{N}
$$
and
$$
\omega\sim u^{\frac{\ell}{a_1}b_1-\frac{\ell}{N}k+(\frac{\ell}{a_1}-1)}\,du\ .
$$
Thus,  the form $\omega$ is holomorphic in a neighborhood of $z_1$ if
and only if the integer
$$
\frac{\ell}{a_1}b_1-\frac{\ell}{N}k+\left(\frac{\ell}{a_1}-1\right)
$$
is nonnegative, which is equivalent to the inequality
$$
b_1+1\ge\frac{a_1}{N}k+\frac{a_1}{\ell}\ ,
$$
which is in turn equivalent to
$$
b_1+1>\frac{a_1}{N}k\ .
$$
Similarly, the conditions
$$
b_i+1>\frac{a_i}{N}k\ ,\quad\text{for } i=1,2,3,4
$$
are  necessary  and  sufficient  for  $\omega$  to  be holomorphic in
neighborhoods   of   $z_1,\dots,z_4$.   The   latter  conditions  are
equivalent to
\begin{equation}
\label{ineq:1}
b_i\ge\left[\frac{a_i}{N}k\right]\ ,\quad\text{for } i=1,2,3,4\ ,
\end{equation}
where $[x]$ denotes the integer part of $x$.

Consider  now  a  coordinate  $v$ in a neighborhood of $\infty$, such
that
$$
z\sim \frac{1}{v}
$$
Then, in a neighborhood of $\infty$ we have
$$
w\sim z^\frac{\sum a_i}{N}\sim v^{-\frac{\sum a_i}{N}}
$$
and
$$
\omega\sim v^{-\sum b_i+\sum a_i\frac{k}{N}-2}\,dv\ .
$$
Thus, the form $\omega$ is holomorphic at $\infty$ if and only if the
integer
$$
-\sum b_i+\sum a_i\frac{k}{N}-2
$$
is  nonnegative,  which  is  equivalent  to  the following inequality
\begin{equation}
\label{ineq:2}
2+\sum_{i=1}^4 b_i\le\sum_{i=1}^4 \frac{a_i}{N}k\ .
\end{equation}
Together with~\eqref{ineq:1} the latter inequality implies
$$
2+\sum_{i=1}^4\left[\frac{a_i}{N}k\right]\ \le\
\sum_{i=1}^4 \frac{a_i}{N}k
$$
or equivalently
$$
\sum_{i=1}^4\left\{\frac{a_i}{N}k\right\}\ \ge\ 2\ .
$$
Passing to the notations~\eqref{eq:tik}, \eqref{eq:tk} we can rewrite
the latter inequality as
\begin{equation}
\label{ineq:3}
t(k)\ \ge\ 2\ .
\end{equation}

Note  that  for  any  $k=1,\dots,N-1$ we have $t(k)\in\{1,2,3\}$. The
inequality~\eqref{ineq:3}  implies  that  for  those  $k$  for  which
$t(k)=1$,  there  are  no  integer  solutions  $b_1,\dots,b_4$ of the
system of inequalities~\eqref{ineq:1}--\eqref{ineq:2}.

For those $k$, for which $t(k)=2$, there is a single integer solution
$b_1,\dots,b_4$           of          the          system          of
inequalities~\eqref{ineq:1}--\eqref{ineq:2},                   namely
$b_i=\left[\frac{a_i}{N}k\right],\ i=1,2,3,4$.

Finally,  when  $t(k)=3$,  there  is  a  ``basic''  integer  solution
$b_i=\left[\frac{a_i}{N}k\right]$,  where  $i=1,2,3,4$, of the system
of   inequalities~\eqref{ineq:1}--\eqref{ineq:2},   and   four  other
solutions,  where  precisely  one  of the integers $b_i$ of the basic
solution  is augmented by one. The resulting holomorphic forms span a
two-dimensional    family    described    in    the    statement   of
Lemma~\ref{lm:basis}.
\end{proof}

\begin{Lemma}
\label{lm:g}
For  any  cyclic  cover  $\MNa$, the holomorphic forms constructed in
Lemma~\ref{lm:basis}  for  $k=1,\dots,N-1$, form a basis of the space
of holomorphic forms.
\end{Lemma}
\begin{proof}
The statement of the Lemma is equivalent to the following identity:
\begin{equation}
\label{eq:to:verify}
\sum_{k=1}^{N-1} \left(t(k)-1\right) = g\ .
\end{equation}
It  is easy to check (see, say,~\cite{Forni:Matheus:Zorich}) that the
genus $g$ of $\MNa$ is expressed in terms of parameters $N,a_i$ as
$$
g=N+1- \frac{1}{2} \sum\limits_{i=1}^4\textrm{gcd}(a_i,N)\ ,
$$
and we can rewrite relation~\eqref{eq:to:verify} as
$$
\sum_{i=1}^4\sum_{k=1}^{N-1} \left\{\frac{a_i}{N}k\right\}=(N-1)+
\big(N+1- \frac{1}{2} \sum\limits_{i=1}^4\textrm{gcd}(a_i,N)\big)\ .
$$
Thus,  to  prove~\eqref{eq:to:verify}  it is sufficient to prove that
\begin{equation}
\label{eq:to:verify:bis}
\sum_{k=1}^{N-1} \left\{\frac{a_i}{N}k\right\}=
\frac{1}{2}
\big(N-  \textrm{gcd}(a_i,N)\big)\ \text{ for }i=1,\dots,4\ .
\end{equation}

For any integers $a$, and $N> 0$, a sequence
$$
\left\{\frac{a}{N}\right\},\
\left\{2\frac{a}{N}\right\},
\dots
$$
is periodic with period $T=N/\gcd(N,a)$. A collection of numbers
$$
\left\{
\left\{\frac{a}{N}\right\},\
\left\{2\frac{a}{N}\right\},
\dots,
\left\{T\frac{a}{N}\right\}
\right\}
$$
within  each  period  considered  as \textit{unordered} set coincides
with the set
$$
\left\{
\frac{0}{T},\frac{1}{T},\frac{2}{T},\dots,\frac{T-1}{T}\right\}
$$
which  forms  an  arithmetic  progression. The sum of numbers in this
latter set equals $(T-1)/2$. Hence,
\begin{multline*}
\sum_{k=1}^{N-1} \left\{\frac{a}{N}k\right\}\ =\
\gcd(N,a)\sum_{k=1}^{T} \left\{\frac{a}{N}k\right\}
\ =\\=\
\frac{\gcd(N,a)}{2}\left(\frac{N}{\gcd(N,a)}-1\right)\ =\
\frac{1}{2}\left(N-  \gcd(a,N)\right)\ ,
\end{multline*}
which proves~\eqref{eq:to:verify:bis}.
\end{proof}

Now     let     us     see    how    relations~\eqref{eq:dim:duality}
and~\eqref{eq:dim:Vk}   of   the   Theorem  of  I.~Bouw  follow  from
Lemmas~\ref{lm:basis} and~\ref{lm:g}.

Equation~\eqref{eq:T}  describing  the  action  of  the group of deck
transformations    implies    that    all    forms   constructed   in
Lemma~\ref{lm:basis}       are      eigenforms,      namely,      the
form~\eqref{eq:general:form}  is  an  eigenform  with  the eigenvalue
$\zeta^{-k}=\zeta^{N-k}$.  By  Lemma~\ref{lm:g}  they form a basis of
the      space      of      holomorphic      forms.     Hence     the
forms~\eqref{eq:general:form} give a basis of $V^{1,0}(N-k)$ for each
individual  $k$.  Combined with Lemma~\ref{lm:basis} this observation
implies
$$
\dim_{\C{}} V^{1,0}(k)=t(N-k)-1\ .
$$

Clearly,        the        complex        conjugate       of       an
eigenform~\eqref{eq:general:form}  is  an  eigenform with a conjugate
eigenvalue.  Hence, the complex conjugates of~\eqref{eq:general:form}
form a basis in the space of antiholomorphic forms and
$$
\dim_{\C{}} V^{1,0}(k)=\dim_{\C{}} V^{0,1}(N-k)\ .
$$
This implies that
$$
\dim_{\C{}} V(k)=t(k)+t(N-k)-2\in\{0,1,2\}\, .
$$

\subsection{Extension  of the Hodge bundle and of its direct summands
to cusps: general approach}
\label{ss:extension:general}

Consider  a  holomorphic family $X_\epsilon$ of smooth complex curves
of  genus  greater  than  one  over  a  punctured  disc, $\epsilon\in
\cD\setminus\{0\}$.  By  the  Deligne--Mumford--Grothendieck  Theorem
this family extends to a holomorphic family of stable curves over the
entire  disc $\cD$ if and only if the induced monodromy on the bundle
$H^1_{\C{}}$ is unipotent.

Geometrically  such  an  extension  can  be described as follows. The
complex  curves  $X_\epsilon$  considered as Riemann surfaces endowed
with  the  hyperbolic  metric  get  pinched  at some fixed collection
$\Gamma=\{\gamma_1,\dots,\gamma_k\}$  of  short  hyperbolic geodesics
such  that for small $\epsilon$ the thick components $X_{\epsilon,j}$
of $X_\epsilon-\Gamma$ are ``geometrically close'' to the irreducible
components  $X_{0,j}$  of  the  stable  curve  $X_0$ endowed with the
canonical  hyperbolic  metric  with  cusps  at  the nodal points. The
limiting curve $X_0$ is a regular smooth complex curve if and only if
and  only if no hyperbolic geodesic gets pinched, so that $\Gamma$ is
empty in this particular case.

The  Deligne extension of the Hodge bundle has the following fiber at
$\epsilon=0$.  It  consists  of  holomorphic  $1$-forms on the smooth
locus  of  the  stable  curve $X_0$ which have simple poles at double
points  of  $X_0$ and such that the sum of two residues at any double
point vanishes.

Now  suppose  that  in addition we are given a nontrivial holomorphic
section  $\omega$  of  the  Hodge bundle $H^{1,0}$ over the punctured
disc       $\cD\setminus\{0\}$,       where       $\omega_\epsilon\in
H^{1,0}(X_\epsilon)$.   According   to   the  generalization  of  the
Deligne--Mumford--Grothendieck  Theorem  for any such section and for
each  irreducible  component  $X_{0,j}$  we  have  an  integer number
$d=d(j,\omega)$     such     that     the     holomorphic    $1$-form
$\epsilon^{-d}\omega_\epsilon$  restricted  to  the  thick  component
$X_{\epsilon,j}$  tends  to  some well-defined nontrivial meromorphic
form $\tilde\omega_{0,j}$ on the desingularized irreducible component
$X_{0,j}$.  The  limiting $1$-form $\tilde\omega_{0,j}$ is allowed to
have  poles  (possibly  of  order greater than one) only at the nodal
points  of $X_{0,j}$, see~\cite{Eskin:Kontsevich:Zorich} for details.
Note  that  in  general  the  integers  $d(j,\omega)$  vary  from one
irreducible  component  $X_{0,j}$  to the other. Define $d(\omega)\in
\Z$  to  be the minimum of integers $d(j,\omega)$ over all components
$X_{0,j}$.      The      limit      at     $\epsilon\to     0$     of
$\epsilon^{-d(\omega)}\omega_\epsilon$  is a form with at most simple
poles  at  double  points,  possibly  vanishing  at  some  components
$X_{0,j}$  of  the special fiber $X_0$, and non-vanishing on at least
one component $X_{0,k}$.

Let  us  assume that the line subbundle of $H^{1,0}$ generated by the
section  $\omega$  is the subbundle $\mathcal{E}^{1,0}$ corresponding
to  a direct summand $\mathcal{E}_{\C{}}$ of the variation of complex
polarized  Hodge structures $H_{\C{}}$, with signature $(1,1)$. Then,
outside  of  the  cusp,  $\omega$  gives  a  section of the canonical
extension of $\mathcal{E}^{1,0}$ to the cusp. The latter extension is
obviously  a  direct summand of the canonical extension of $H^{1,0}$.
Therefore,   the   order   of   zero   of  the  section  $\omega$  of
$\overline{\mathcal{E}^{1,0}}$    at   $\epsilon=0$   is   equal   to
$d(\omega)$.

Finally,  if  our family does not allow a semi-stable reduction (i.e.
the monodromy is not unipotent), then one can pass to a finite cyclic
cover of certain order $n\ge 1$ of the punctured disc, and reduce the
question to the unipotent case. Let us denote by $\tilde{\omega}$ the
pullback  of  the section $\omega$ to the covering. Then the orbifold
order  of  vanishing  at $\epsilon \to 0$ of $\omega$ is the order of
zero  $d(\tilde{\omega})$  corresponding  to the covering, divided by
$n$.

One  can  replace  this  formula by a ``local'' one. Namely, for each
component  $X_{0,j}$  of  a  special  fiber of a semistable reduction
obtained  by  passing  to  a  finite  covering,  we  associate a {\it
rational}  number  $r(j,\omega)=d(j,\tilde{\omega})/n$  equal  to the
fractional  order  of  the  zero in $\epsilon$, and then consider the
minimum over all components.

After  completing the paper~\cite{Eskin:Kontsevich:Zorich} we learned
that analogous results were simultaneously and independently obtained
in~\cite{Grushevsky:Krichever}     by     \mbox{S.~Grushevsky}    and
I.~Krichever   and   in~\cite{Calta:Schmidt:Smillie}   by   K.~Calta,
T.~Schmidt and J.~Smillie.

\subsection{Extension  of  the line bundles $V^{1,0}(k)$ to the cusps
of the Teichm\"uller curve}
\label{ss:extension:of:V}

Before passing to the proof of Theorem~\ref{th:degree} let us discuss
how  the general setting described in the previous section applies to
our concrete situation.

Recall  that  by Convention~\ref{conv:naming:ramification:points} all
the points $z_i$ are ``named'' and, thus, distinguishable. Hence, the
arithmetic  Teichm\"uller curve $\cC=\cM_{(a_i),N}$ can be identified
with the space $\cM_{0,4}$ of configurations of ordered quadruples of
distinct   points   $\{z_1,z_2,z_3,z_4\}$   on   the  Riemann  sphere
considered  up to a holomorphic automorphism preserving the ``names''
(see~\cite{Forni:Matheus:Zorich}  for  details).  We fix three points
$z_1,  z_2,  z_4$  equal  to  $0,1,\infty$  correspondingly. Now, our
cyclic     cover     is    defined    by    a    complex    parameter
$z_3\not\in\{0,1,\infty\}$  which  serves  as  a  coordinate  on  the
compactified                    Teichm\"uller                   curve
$\overline{\cC}\simeq\overline{\cM_{0,4}}\simeq\CP$.  This curve is a
base  of the universal curve. Under the chosen normalization, a fiber
$X_{z_3}$  of the universal curve over a point $z_3\in\cC$ is defined
by equation
\begin{equation}
\label{eq:curve:0:1:infty}
w^N=z^{a_1}(z-1)^{a_2}(z-z_3)^{a_3}\,.
\end{equation}

\begin{figure}
\includegraphics{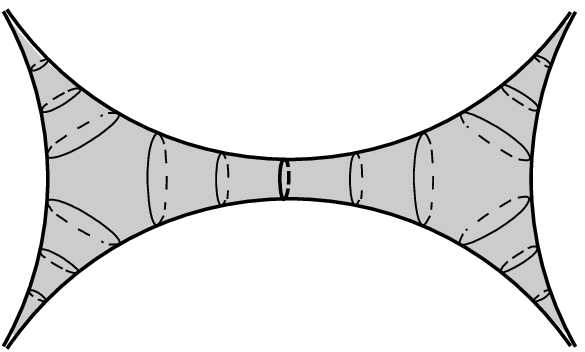}
\includegraphics{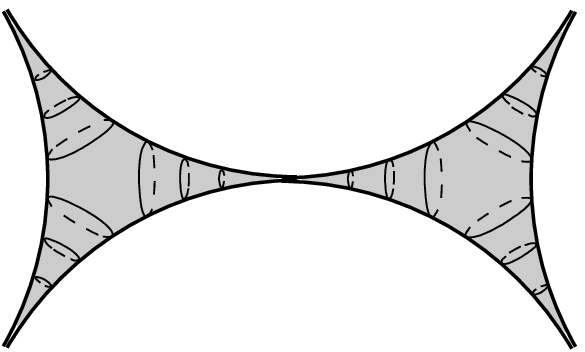}
\begin{picture}(0,0)(0,0)
\put(27,-15){$z_1$}
\put(27,-108){$z_2$}
\put(162,-108){$z_3$}
\put(162,-15){$z_4$}
\put(99,-115){b.}
\end{picture}
\begin{picture}(0,0)(194,0)
\put(27,-15){$z_1$}
\put(27,-108){$z_2$}
\put(162,-108){$z_3$}
\put(162,-15){$z_4$}
\put(98,-115){a.}
\end{picture}
\vspace{115bp}
\caption{
\label{fig:four:cusps:sphere}
Degeneration of $\CP$ with $4$ marked points.}
\end{figure}

We  warn the reader of a possible confusion: we have $\CP$ constantly
appearing  in  two different roles. On the one hand the Teichm\"uller
curve  $\cC$  is  isomorphic  to $\CP$ with three cusps. On the other
hand,  the  fiber $X_{z_3}$ of the universal curve over $\cC$ has the
structure of a cyclic cover, $X_{z_3}\simeq\MNa\to\CP\simeq Y_{z_3}$,
and  we have $\CP\simeq Y_{z_3}$ in the base of this cover. We mostly
work with $\CP$ in its second appearance: it can be viewed as a fiber
of the associated universal curve.

Our  goal  is  to calculate the orbifold degree of holomorphic vector
bundle  $V^{1,0}(k)$ on the curve $\cM_{(a_i),N}\simeq \cM_{0,4}$, in
the  case  when  it  is  a  {\it  line} bundle. Recall that we have a
canonical  section  of $V^{1,0}(k)$ given by one of basis elements of
$H^{1,0}$,  hence the degree coincides with the sum of multiplicities
of  zeroes  of this section. All zeroes and poles appear only at cusp
points.

A  point $z_3\in\cC$ of the Teichm\"uller curve approaches one of the
three  cusps  when  $z_3$  tends  to  $0, 1$ or $\infty$. The induced
monodromy  on  the bundle $H^1_{\C{}}$ is not unipotent, therefore we
should  go to an appropriate cover $\cC'\to \cC$ of the Teichm\"uller
curve  (e.g.  the  punctured  Fermat  curve, see the last sentence in
\ref{ss:theorem:sum}).   The   exact   nature  of  this  cover  is
irrelevant.

The  stable curve which arises as the limit as $z_3$ goes to the cusp
on $\overline{\cC}$ is a cover (possibly ramified at the double point
and  marked  points)  of  a semistable genus zero curve with 4 marked
points,  which  is  the limit of $Y_{z_3}$. The stable curves (with 4
marked  points) $Y_0$, $Y_1$ and $Y_\infty$ over the three cusps have
the  same structure, see Figure~\ref{fig:four:cusps:sphere}b. Each of
them has two components, where each component is a $\CP$ endowed with
three  marked  points;  the  two  components  are  glued  together by
identifying a pair of the marked points.

Geometrically  the  degeneration of the fiber $Y_{z_3}$ near the cusp
can  be  described  as follows. The fiber $Y_{z_3}\simeq\CP$ has four
marked points, namely, $0,1,\infty,z_3$. The corresponding hyperbolic
surface  is  a  topological sphere with four cusps (a cusp at each of
the  marked  points).  Such  hyperbolic surface can be glued from two
identical  pairs of pants, where each pair of pants has two cusps and
a  nontrivial  geodesic  boundary  curve;  the  boundaries  are glued
together        by       a       hyperbolic       isometry,       see
Figure~\ref{fig:four:cusps:sphere}a. When $z_3$ approaches one of the
points  $0,1,\infty$  the  close  hyperbolic geodesic, serving as the
common  waist  curve  of the pairs of pants, gets pinched, and at the
limit  we  get two identical pairs of pants, each having three cusps,
see~Figure~\ref{fig:four:cusps:sphere}b.

By  assumption  of  Theorem~\ref{th:degree}  we  consider  only those
values    of    $k$,    for    which    one    has   $t(N-k)=2$.   By
formula~\eqref{eq:dim:duality}  from  the  Theorem  of  I.~Bouw, this
implies  that  $\dim_{\C{}}V^{1,0}(k)=1$. By Lemma~\ref{lm:basis} the
fiber  $l_{z_3}$  of  the corresponding line bundle $l=V^{1,0}(k)$ is
spanned by the holomorphic form
\begin{equation}
\label{eq:form:0:1:infty}
\omega=z^{b_1}(z-1)^{b_2}(z-z_3)^{b_3}\cfrac{dz}{w^k}\ ,\qquad
           b_i=[a_i\cdot k/N],\,i=1,2,3\ .
\end{equation}

To  compute the orbifold degree of $l=V^{1,0}(k)$ it is sufficient to
compute  the  divisor  of  the section $\omega$. Since outside of the
cusps the section $\omega$ is nonzero, we have to compute the degrees
of  zeroes  or  poles of the extension of the section $\omega$ at the
three cusps $0,1,\infty$ of the Teichm\"uller curve $\cC$.

Following  the  general approach presented in the previous section we
introduce  a  small local parameter $\epsilon$ in a neighborhood of a
cusp  on  $\cC$. Let, for example, the cusp be presented by the point
$0$.  For  each component $X_{0,j}$ of the stable curve $X_0$ we need
to  find  an  appropriate  fractional  power  $r(j,\omega)\in  \Q$ of
$\epsilon$ such that the form $\epsilon^{-r(j,\omega)}\cdot\omega$ on
the  corresponding  thick part $X_{\epsilon,j}$ of $X_\epsilon$ tends
to a nontrivial meromorphic form on the chosen group of components of
the stable curve $X_{0,j}$ as $\epsilon\to 0$.

As  it will visible from the calculation we, actually, do not need to
consider  individually all irreducible components of the stable curve
$X_0$.   We   use   the  projection  $X_0\to  Y_0$  map  between  the
corresponding  stable curves to organize the components of $X_0$ into
two   groups   corresponding  to  preimages  of  the  two  components
$Y_{0,1},Y_{0,2}$     of     the     stable    curve    $Y_0$,    see
Figure~\ref{fig:four:cusps:sphere}b.   All   components   $X_{0,j_1},
\dots,    X_{0,j_m}$    in    the   same   group   corresponding   to
$Y_{0,\alpha},\ \alpha\in\{1,2\},$ share  the  same  rational  number
$r^{(\alpha)}:=r(j_1,\omega)=\dots=r(j_m,\omega)$.     The    minimum
$\min(r^{(1)},r^{(2)})$   of   the  resulting  two  rational  numbers
corresponding  to  the  two  components  $Y_{0,1},  Y_{0,2}$ of $Y_0$
determines  the  order  of  a zero or pole of the meromorphic section
$\omega$ of $V^{1,0}$ at the cusp $0$ in $\cC$. The situation for the
two   remaining  cusps  $1$  and  $\infty$  in  $\cC$  is  completely
analogous.

\subsection{Computation of degrees of line bundles}
\label{ss:degree}

In  this  section  we  prove  Theorem~\ref{th:degree}.  The notations
$t_i(k)$  and  $t(k)$  which  we  use in the proof were introduced in
equations~\eqref{eq:tik} and~\eqref{eq:tk} correspondingly.

\begin{proof}[Proof of Theorem~\ref{th:degree}]
Let,     $z_3=\epsilon\to    0$.    The    base    of    the    cover
$X_{z_3}\simeq\MNa\to\CP$  splits into two Riemann spheres, where $0$
and  $z_3$  stay in one component, and $1$ and $\infty$ belong to the
other one,
see  Figure~\ref{fig:four:cusps:sphere}b.

Introducing coordinates
$\tilde z=\epsilon^{-1} z$ and
$\tilde w=\epsilon^{-\frac{a_1+a_3}{N}} w
$
in the first component we see that
$$
\omega=\epsilon^{(b_1+b_3+1-\frac{a_1+a_3}{N}k)}\,
\tilde\omega_\epsilon\ ,
$$
where   a   holomorphic   form  $\tilde\omega_\epsilon$  tends  to  a
nontrivial  meromorphic  form  $\tilde\omega_0$ when $z_3=\epsilon\to
0$. Here and below ``trivial'' means everywhere null.

\begin{NNRemark}
Formally,  we  should introduce $n:=N/\gcd(a_1+a_3, N)$ and pass to a
ramified  $n$-fold  cover  $\hat\cD\to\cD$ of a neighborhood $\cD$ of
the  cusp  $0$  at  $\cC$,  so that $\epsilon=\delta^n$. Then $\tilde
z=\delta^{-n}z$   and  $\tilde  w=\delta^{-\frac{(a_1+a_3)n}{N}}  w$,
where  the  powers  of $\delta$ are already integer; see the comments
about    the   orbifold   degree   in   sections~\ref{ss:theorem:sum}
and~\ref{ss:extension:general}.
\end{NNRemark}

On the other component $\omega$ tends to a nontrivial form when $z_3$
tends  to  $0$  without  any  renormalization. Hence, our section has
singularity of order
$\min\left(0,b_1+b_3+1-\cfrac{a_1+a_3}{N}\,k\right)\ $
at this cusp. We proved in Lemma~\ref{lm:basis} that when $t(N-k)=2$,
we  have  $b_i=\left[\frac{a_i}{N}k\right]$. Thus, we can rewrite the
above expression as
\begin{equation}
\label{eq:deg:0}
\min\left(0\,,\,b_1+b_3+1-\cfrac{a_1+a_3}{N}\,k\right)=
\min\big(0\,,\,1-(t_1(k)+t_3(k))\big)\ .
\end{equation}

Similarly,  the order of the singularity of the section $\omega$ at
the cusp $z_3=1$ is equal to
\begin{equation}
\label{eq:deg:1}
\min\big(0\,,\,1-(t_2(k)+t_3(k))\big)\ .
\end{equation}

Finally,  when $z_3\to\infty$ the base of the cover $\MNa\to\CP$ once
again  splits into two Riemann spheres, where $0$ and $1$ stay in one
component,  and $z_3$ and $\infty$ belong to the other component. Let
$z_3=\frac{1}{\epsilon}$,  and  consider  the behavior of $\omega$ on
the  first  component, containing $0$ and $1$, as $\epsilon$ tends to
$0$.

The  equation~\eqref{eq:curve:0:1:infty}  of  the  curve implies that
$w\sim\epsilon^{-\frac{a_3}{N}}$.                                 The
equation~\eqref{eq:form:0:1:infty} of the form implies that
$$
\omega=\epsilon^{(-b_3+\frac{a_3}{N}k)}\,
\phi_\epsilon\ ,
$$
where  the  meromorphic  form  $\phi_\epsilon$  tends to a nontrivial
meromorphic    form    $\phi_0$   on   the   first   component   when
$z^{-1}_3=\epsilon\to 0$.

Introduce  the  coordinate  $\tilde  z=\frac{1}{\epsilon  z}$  on the
second  component.  The  equation~\eqref{eq:curve:0:1:infty}  of  the
curve   implies  that  $w\sim\epsilon^{-\frac{a_1+a_2+a_3}{N}}$.  The
equation~\eqref{eq:form:0:1:infty} of the form implies that
$$
\omega=\epsilon^{(-b_1-b_2-b_3-1+\frac{a_1+a_2+a_3}{N}k)}\,
\psi_\epsilon\ ,
$$
where  the  meromorphic  form  $\psi_\epsilon$  tends to a nontrivial
meromorphic    form   $\psi_0$   on   the   second   component   when
$z^{-1}_3=\epsilon\to 0$.

By  Lemma~\ref{lm:basis}  we  have $b_i=\left[\frac{a_i}{N}k\right]$.
Applying           the          notations~\eqref{eq:tik}          for
$t_i(k)=\left\{\frac{a_i}{N}k\right\}$  we  conclude that the section
$\omega$ of the line bundle $V^{1,0}(k)$ has singularity of order
$$
\min\big(t_3(k)\,,\,t_1(k)+t_2(k)+t_3(k)-1\big)
$$
at $\infty$.

Taking  into consideration that $t(k)=t_1(k)+t_2(k)+t_3(k)+t_4(k)=2$,
we can rewrite the latter expression as
\begin{multline}
\min\big(t_3(k)\,,\,t_1(k)+t_2(k)+t_3(k)-1\big)
\ =\\=
t_3(k)+\min\big(0,t_1(k)+t_2(k)-1\big)
=
\label{eq:deg:infty}
t_3(k)+
\min\big(0,1-(t_3(k)+t_4(k))\big)\,.
\end{multline}

Summing             up~\eqref{eq:deg:0},             \eqref{eq:deg:1}
and~\eqref{eq:deg:infty}  we  see  that  the  orbifold  degree of the
section $\omega$ is equal to
\begin{multline}
\label{eq:deg}
t_3(k)\ +\
\min\big(0\,,\,1-(t_1(k)+t_3(k))\big)
\ +\\+\
\min\big(0\,,\,1-(t_2(k)+t_3(k))\big)\ +\
\min\big(0\,,\,1-(t_4(k)+t_3(k))\big)
\end{multline}

Formula~\eqref{eq:dk:equals:min}  for the orbifold degree of the line
bundle  $V^{1,0}(k)$  stated  in  Theorem~\ref{th:degree} now follows
from  the relation~\eqref{eq:identity} proved in the elementary Lemma
below.

\begin{Lemma}
For  any  quadruple of numbers $t_i\in[0,1[$, satisfying the relation
$t_1+t_2+t_3+t_4=2$, the following identity holds:
\begin{multline}
\label{eq:identity}
\min\big(t_1,1-t_1,\dots,t_4,1-t_4\big)
\ =\\=
t_3 +
\min\big(0,1-(t_1+t_3)\big)+
\min\big(0,1-(t_2+t_3)\big)+
\min\big(0,1-(t_4+t_3)\big)
\end{multline}
\end{Lemma}
\begin{proof}
First note, that expression~\eqref{eq:deg}, is symmetric with respect
to  permutations  of  indices  $1,2,3,4$,  just  because  the initial
geometric setting is symmetric. Since the corresponding quadruples of
numbers  $t_i(k)$, taken for different data $N,a_1,\dots,a_4$, form a
dense  set  in $[0,1]^4$, we conclude that the continuous function in
the  right-hand side of~\eqref{eq:identity} is symmetric with respect
to   permutations  of  indices  $1,2,3,4$.  Hence,  without  loss  of
generality  we may assume that $t_1\le t_2\le t_3\le t_4$. Under this
assumption the left-hand side expression in~\eqref{eq:identity} takes
the value $\min(t_1,1-t_4)$.

Consider the expression in the right-hand side. Note that $t_2+t_4\ge
t_1+t_3$.  Since $t_2+t_4 + t_1+t_3=2$, this implies that $t_1+t_3\le
1$, and hence
$$
\min\big(0\,,\,1-(t_1+t_3)\big)=0\ .
$$
Similarly, since $t_3+t_4\ge t_1+t_2$ we conclude that $t_3+t_4\ge 1$
and hence
$$
\min\big(0\,,\,1-(t_4+t_3)\big)=1-(t_4+t_3)\ .
$$

The  value  of  the  middle  term depends on comparison of $t_1$ with
$1-t_4$. If $t_1\le 1-t_4$, then $t_1+t_4\le 1$ and hence $t_2+t_3\ge
1$.  In  this  case  we  get  the  value $t_1$ for the left-hand side
expression in~\eqref{eq:identity} and
$$
t_3+0+(1-(t_2+t_3))+(1-(t_4+t_3))=2-((t_2+t_3+t_4)=t_1
$$
as the value of the right-hand side expression.

If  $t_1>  1-t_4$,  then $t_1+t_4> 1$ and hence $t_2+t_3< 1$. In this
case  we  get  the  value  $1-t_4$  for the left-hand side expression
in~\eqref{eq:identity} and
$$
t_3+0+0+(1-(t_4+t_3))=1-t_4
$$
as  the value of the right-hand side expression. The desired identity
is proved.
\end{proof}

 To complete the proof of Theorem~\ref{th:degree} it
remains  to  prove  the  alternative  concerning positivity of $d(k)$
claimed in the statement of the Theorem.

By  assumptions  of Theorem~\ref{th:degree} we have
$t(N-k)=2$.  By  formula~\eqref{eq:dim:duality}  from  the Theorem of
I.~Bouw  one  gets  $\dim_{\C{}} V^{1,0}(k) = 1$. Hence, $\dim_{\C{}}
V(k)\ge  1$.  Thus,  formula~\eqref{eq:dim:Vk}  from  the  Theorem of
I.~Bouw implies that $t(k)$ is equal either to $2$ or to $1$.

If  $t(k)=2$, then by formula~\eqref{eq:dim:Vk} one gets $\dim_{\C{}}
V(k)  =  2$.  By  Lemma~\ref{lm:dim:Vk:equals:two}  this implies that
$t_i(k)>0$  for  $i=1,2,3,4$.  Since  by definition~\eqref{eq:tik} of
$t_i(k)$  one always has $t_i(k)<1$, formula~\eqref{eq:dk:equals:min}
implies that $d(k)>0$ in this case.

If  $t(k)=1$, then by formula~\eqref{eq:dim:Vk} one gets $\dim_{\C{}}
V(k)  =  1$.  By  Lemma~\ref{lm:dim:Vk:equals:two}  this implies that
there  is  an index $i\in\{1,2,3,4\}$, such that $t_i(k)=0$. Thus, in
this case formula~\eqref{eq:dk:equals:min} implies that $d(k)=0$.

Theorem~\ref{th:degree} is proved.
\end{proof}


\appendix
\section{Hodge bundles associated to quadratic differentials}
\label{a:quadratic}

Consider  a  cyclic cover $p:\MNa\to \CP$ and a meromorphic quadratic
differential  $p^\ast  q_0$  on $\MNa$ defining the flat structure on
$\MNa$.  Here  a  canonical  quadratic differential $q_0$ on $\CP$ is
defined  by  equation~\eqref{eq:q:on:CP1}. The quadratic differential
$q$  is a global square of an Abelian differential if and only if $N$
is even and all $a_i$ are odd, see~\cite{Forni:Matheus:Zorich}.

Suppose  that  at least one of the following conditions is valid: $N$
is odd or one of $a_i$ is even. Then $q$ is not a global square of an
Abelian  differential.  There  exists a canonical (possibly ramified)
double cover $p_2:\hat X\to \MNa$ such that $p_2^\ast q=\omega^2$,
where $\omega$ is already a holomorphic $1$-form on $\hat X$.

Let $\hat g$ be the genus of the cover $\hat X$. By \textit{effective
genus} we call the positive integer
\begin{equation}
\label{eq:g:eff}
g_{\mathit{eff}}:=\hat g - g\ .
\end{equation}

The  cohomology  space  $H^1(\hat  X,\R{})$  splits into a direct sum
$H^1(\hat  X,\R{})=H_+^1(\hat  X,\R{})\oplus  H_-^1(\hat  X,\R{})$ of
invariant  and  antiinvariant  subspaces  with respect to the induced
action  of  the  involution  $p_2^\ast : H^1(\hat X,\R{})\to H^1(\hat
X,\R{})$.  Note  that the invariant part is canonically isomorphic to
the  cohomology  of the underlying surface, $H_+^1(\hat X,\R{})\simeq
H^1(X,\R{})$.   We   consider   subspaces  $H_+^1(\hat  X,\R{})$  and
$H_-^1(\hat  X,\R{})$ as fibers of natural vector bundles $H^1_+$ and
$H^1_-$  over  the  Teichm\"uller  curve $\cC$. The bundle $H^1_+$ is
canonically isomorphic to the Hodge bundle $H^1$ considered above.

The  splitting  $H^1=H^1_+\oplus  H^1_-$ is covariantly constant with
respect   to   the   Gauss--Manin  connection.  The  symplectic  form
restricted  to  each summand is nondegenerate. Thus, the monodromy of
the  Gauss--Manin  connection  on $H_-^1$ is symplectic. The Lyapunov
exponents  of the bundle $H_-^1$ with respect to the geodesic flow on
$\cC$                  are                 denoted                 by
$\lambda^-_1\ge\dots\ge\lambda^-_{2g_{\mathit{eff}}}$.  As  before we
have  $\lambda^-_k=-\lambda^-_{2g_{\mathit{eff}}-k+1}$. It is natural
to  study  the  Lyapunov spectrum $\operatorname{{\Lambda}Spec}_-$ of
$H^1_-$.

When $N$ is odd, the canonical double cover $\tilde X$ over $\MNa$ is
again   a   cyclic   cover.   Namely,   it   is   the   cyclic  cover
$M_{2N}(a'_1,\dots,a'_4)$  where  $a'_i:=a_i$,  when $a_i$ is odd and
$a'_i:=a_i+N$,     when    $a_i$    is    even.    We    can    apply
Theorem~\ref{th:main:theorem} to compute the Lyapunov spectrum of the
Hodge  bundle  over  $M_{2N}(a'_1,\dots,a'_4)$. By construction, this
Lyapunov  spectrum  is  a union $\operatorname{{\Lambda}Spec}_-\sqcup
\operatorname{{\Lambda}Spec}_+$  of  Lyapunov  spectra of the bundles
$H^-$        and        $H^+$        correspondingly.        Applying
Theorem~\ref{th:main:theorem}  once  more  we  compute  the  Lyapunov
spectrum   $\operatorname{{\Lambda}Spec}_+$   of   the  Hodge  bundle
corresponding  to  $\MNa$.  Taking  the complement of two spectra, we
obtain the Lyapunov spectrum $\operatorname{{\Lambda}Spec}_-$.

When  $N$  is even, and a pair of $a_i$ is even, the canonical double
cover  $\tilde  X$ is not a cyclic cover. However, in this case it is
an \textit{Abelian} cover, see~\cite{Wright}.

\section{Square-tiled cyclic covers with symmetries}
\label{s:symmetries}

Following   I.~Bouw   and   M.~M\"oller,  see~\cite{Bouw:Moeller}  we
consider in this section a situation when a cyclic cover has an extra
symmetry.  Passing  to  a quotient over such an invariant holomorphic
automorphism  we  get  a  new  square-tiled surface and corresponding
arithmetic  Teichm\"uller  curve  for which we can explicitly compute
the Lyapunov spectrum.


The  fact  that the Lyapunov spectrum corresponding to the ``stairs''
square-tiled surfaces discussed below forms an arithmetic progression
(see  Proposition~\ref{pr:palindrome}  in  the  current  section) was
noticed  by  the  authors in computer experiments about a decade ago.
Recently  M.~M\"oller suggested that this property of the spectrum is
a strong indication that the corresponding surface is a quotient of a
cyclic  cover  over  an  automorphism  (by  analogy with examples of
non-arithmetic   Veech  surfaces  discovered  in~\cite{Bouw:Moeller},
having the same property of the spectrum).

The    discussion   with   J.-C.~Yoccoz   on   possible   values   of
\mbox{Lyapunov}   exponents   of   $\SL$-cocycles   ``over  continued
fractions''  was  a  strong  motivation  for  us  to prove a complete
integrability  in  this  example.  We  show  that  the  Hodge  bundle
$H^1_{\R{}}$   decomposes   into  a  direct  sum  of  two-dimensional
covariantly  constant  subbundles,  and  that  such subbundles over a
small  arithmetic  Teichm\"uller  curve  might  have  arbitrary small
Lyapunov exponents.


Consider  a  cyclic  cover  of  the  type $M_N(a,N-a,b,N-b)$. In this
section we
always  assume  that  $N$  is  even and all $a_i$ are odd, so $p^\ast
q_0=\omega^2$. By convention~\ref{conv:naming:ramification:points}
the  four ramification points $P_1,P_2,P_3,P_4$ of the Riemann sphere
$\CP$  are  named,  so  even  when $a=b$ we can distinguish preimages
corresponding   to   $P_1$  and  to  $P_3$.  There  exists  a  unique
holomorphic involution of $\CP$ with interchanging the points in each
of  the  two  pairs  $P_1,P_2$ and $P_3,P_4$ of the marked points. By
construction  it preserves the quadratic differential determining the
flat  structure  on  $\CP$.  This  involution  induces  a holomorphic
involution $\tau$ of the square-tiled cyclic cover $M_N(a,N-a,b,N-b)$
preserving the flat structure in the sense that
$$
\tau^\ast\omega^2=\omega^2\ .
$$
By construction the involution $\tau$ also preserves the structure of
the  cover,  moreover,  it interwinds a generator $T$ of the group of
deck transformations with its inverse:
\begin{equation}
\label{eq:involution}
\tau T=T^{-1}\tau  \ .
\end{equation}
The  latter relation implies that the holomorphic automorphism of the
cyclic  cover  defined  as $\tau_2:=\tau\circ T$ is also a holomorphic
involution,  and  that  it  also  preserves the flat structure in the
sense that
$$
\tau_2^\ast\omega^2=\omega^2\ .
$$
Note that the latter relation implies that
$$
\tau_2^\ast\omega=\pm\omega\ .
$$
Note also that
$$
T^\ast\omega=-\omega\,.
$$
Hence, if one of the two involutions does not preserve $\omega$, then
the  other  does.  Thus,  up  to the interchange of notations for the
involutions, we may assume that
$$
\tau^\ast\omega=\omega\ .
$$

\begin{Proposition}
\label{pr:symmetry}
Let  $N$  be  even  and  let  $a$  and  $b$  be  odd. A quotient of a
square-tiled  cyclic  cover  of  the form $M_N(a,N-a,b,N-b)$ over the
involution  $\tau$  as above is a connected square-tiled surface. The
Lyapunov  spectrum $\operatorname{{\Lambda}Spec}$ of the Hodge bundle
over  an  arithmetic  Teichm\"uller curve of the initial square-tiled
cyclic  cover  can  be  obtained by taking two copies of the Lyapunov
spectrum  $\operatorname{{\Lambda}Spec_+}$  of  the Hodge bundle over
the  arithmetic  Teichm\"uller  curve  of  the  quotient square-tiled
surface and suppressing one of the two entries ``$1$''.
\end{Proposition}
By convention, by Lyapunov spectrum $\operatorname{{\Lambda}Spec}$ we
call  the  top  $g$ Lyapunov exponents, where $g$ is the genus of the
flat surface under consideration.

The  first  statement  of the Proposition is a particular case of the
following elementary Lemma.
\begin{Lemma}
Let  $\sigma:  S\to  S$  be an automorphism of a square-tiled surface
preserving  an  Abelian differential $\omega$, which defines the flat
structure: $\sigma^\ast\omega=\omega$.

The  quotient surface $S/\sigma$ is a connected square-tiled surface.
\end{Lemma}
\begin{proof}
First  note,  that  if a flat surface $(S,\omega)$ admits a tiling by
unit  squares,  such  tiling is unique. (By convention, when $S$ is a
torus, it is endowed with a marked point serving as a ``fake zero''.)
Hence, an automorphism $\sigma$ maps squares of the tiling to squares
of  the  same  tiling by parallel translations. This implies that the
quotient surface is square-tiled.

Let  us  enumerate  the  squares  by  numbers  $1,\dots,M$. Denote by
$\pi_h,\pi_v$  permutations  indicating  the  squares adjacent to the
right (correspondingly atop) to every square of the tiling.

Consider  some  nonsingular  component  of  the quotient surface. Let
$A\subseteq\{1,\dots,M\}$  denote  the  subset  of all squares of the
initial  surface $S$ which project to this nonsingular component. The
subset  $A$  is  invariant  under both $\pi_h$ and $\pi_v$. Since the
initial  surface is nondegenerate, it implies that $A$ coincides with
the entire set $\{1,\dots,M\}$.
\end{proof}

\begin{Remark}
We  have,  actually,  shown  in  the  proof of the Lemma above that a
square-tiled  surface  $(S,\omega)$  admits a nontrivial automorphism
$\sigma:  S\to  S$, such that $\sigma^\ast\omega=\omega$, if and only
if  there  exists a nontrivial permutation $\sigma$ of squares of the
tiling, which commutes with both $\pi_h$ and $\pi_v$.
\end{Remark}

\begin{proof}[Proof of Proposition~\ref{pr:symmetry}]
It  remains  to prove the statement concerning the Lyapunov spectrum.
By  construction  of  the  involution  $\tau$, the induced involution
$\tau^\ast$  on cohomology interchanges the eigenspaces corresponding
to eigenvalues $\zeta^k$ and $\zeta^{-k}=\zeta^{N-k}$:
\begin{equation}
\label{eq:action:tau}
\tau^\ast V(k)=V(N-k)\qquad\qquad \tau^\ast V^{1,0}(k)= V^{1,0}(N-k)\ ,
\end{equation}
and  preserves the subspaces $W(k)\subset H^1(X,\R{})$ defined in the
beginning of section~\ref{ss:Lyapunov:spectrum:for:cyclic:covers}.

For  any  $k$  such that $1\le k< N/2$ we can represent the subspaces
$V(k)\oplus  V(N-k)$,  $V^{1,0}(k)\oplus  V^{1,0}(N-k)$ and $W(k)$ as
direct  sums  of  invariant  and  anti-invariant  subspaces under the
involution $\tau^\ast$:
\begin{align*}
V(k)\oplus V(N-k)&=V_+(k)\oplus V_-(k)\\
V^{1,0}(k)\oplus V^{1,0}(N-k)&=V^{1,0}_+(k)\oplus V^{1,0}_-(k)\\
W(k)&=W_+\oplus W_-
\end{align*}
Since    for   $k<N/2$   one   has   $V(k)\cap   V(N-k)=\{0\}$,   the
relations~\eqref{eq:action:tau}  imply  that  for $k<N/2$ each of the
subspaces  $V_+(k)$ and $V^{1,0}_+(k)$ has half of a dimension of the
corresponding    ambient    subspace    $V(k)\oplus    V(N-k)$    and
$V^{1,0}(k)\oplus   V^{1,0}(N-k)$.   By  construction,  the  subspace
$V_+(k)$  is invariant under the complex conjugation. Hence, the real
subspace  $W_+(k)$  has  half  of  a  dimension  of  the ambient real
subspace  $W(k)$.  Moreover, $W_+(k)$ is a symplectic $\SL$-invariant
and Hodge star-invariant subspace.

It  follows from the Theorem~\ref{th:main:theorem} that if the vector
bundle  $W(k)$,  where  $k<N/2$,  has  at  least one nonzero Lyapunov
exponent,  the  bundle $W(k)$ necessarily has dimension four, and its
Lyapunov           spectrum           has           the          form
$\{\lambda,\lambda,-\lambda,-\lambda\}$. Hence, the Lyapunov spectrum
of   the   subbundle   $W_+(k)$   of   such   bundle   has  the  form
$\{\lambda,-\lambda\}$.

Note also that by construction the subspaces $V(N/2)$, $V^{1,0}(N/2)$
and,  hence,  $W(N/2)$ are invariant under the involution $\tau^\ast$
which   acts   on   these   subspaces   as  an  identity  map.  Thus,
$V_+(N/2)=V(N/2)$,          $V_+^{1,0}(N/2)=V^{1,0}(N/2)$         and
$W_+(N/2)=W(N/2)$.  Recall  that the Lyapunov spectrum of $W(N/2)$ is
$\{1,-1\}$.

Consider now a square-tiled surface $S$ obtained as a quotient of the
square-tiled cyclic cover $M_N(a,N-a,b,N-b)$ over $\tau$. Clearly, we
have canonical isomorphisms
\begin{align}
\label{eq:quotient:decomposition}
H^1(S,\C{})&\simeq \bigoplus_{k\le N/2}V_+(k)\notag\\
H^{1,0}(S)&\simeq \bigoplus_{k\le N/2}V^{1,0}_+(k)\\
H^1(S,\R{})&\simeq \bigoplus_{k\le N/2}W_+(k)\notag
\end{align}
Taking  into  consideration  our observations concerning the Lyapunov
spectrum  of  the  subbundles  $W(k)$,  this implies the statement of
Proposition~\ref{pr:symmetry}  concerning  the  spectrum  of Lyapunov
exponents of the quotient surface.
\end{proof}

Let us consider in more detail two particular cases.
\medskip

\subsection{Lyapunov spectrum of ``stairs'' square-tiled surfaces}
\label{a:stairs}

Consider  a  following  square-tiled surface $S(N)$. Take $N$ squares
and arrange them cyclically into a cylinder of width $N$ and of hight
$1$.  A  permutation  $\pi_h$,  which indicates a right neighbor of a
square number $k$, is given by a cycle $(1,\dots,N)$. Now identify by
a  parallel  translation the top horizontal side of the square number
$k$  to  the  bottom  horizontal side of the square number $N+1-k$. A
permutation $\pi_v$, which indicates a neighbor above a square number
$k$, is as follows:
$$
\pi_v=
\begin{pmatrix}
N & N-1 & \dots &  2  & 1\\
1   &  2  & \dots & N-1 & N
\end{pmatrix}
$$
The     resulting    square-tiled    surface    is    presented    at
Figure~\ref{fig:palindrome}.

\begin{figure}[hb]
   %
\includegraphics{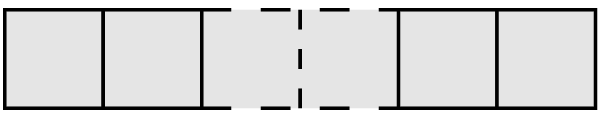}
\begin{picture}(0,0)(60,-7) 
\put(10,-25){1}
\put(30,-25){2}
\put(46,-25){$\cdots$}
\put(66,-25){$\cdots$}
\put(84,-25){N-1}
\put(110,-25){N}
\put(10,-10){\tiny N}
\put(27,-10){\tiny N-1}
\put(91,-10){\tiny 2}
\put(111,-10){\tiny 1}
\end{picture}
\vspace{20bp}
\caption{
\label{fig:palindrome}
Square-tiled surface $S(N)$.}
\end{figure}

\begin{NNConvention}
Representing a square-tiled surface by a polygonal pattern we usually
try  to  respect  a decomposition into horizontal cylinders. Thus, by
convention,   an   unmarked  vertical  segment  of  the  boundary  is
identified  by  a parallel translation with another unmarked vertical
segment  of  the  boundary  located  at the same horizontal level. To
specify  identification of the horizontal segments of the boundary we
give  a  number  of  a  square located atop of each square of the top
boundary.  When  such  square  is  not  indicated,  it  means that we
identify   the   two  horizontal  segments  by  a  vertical  parallel
translation.

Note  that  a  general square-tiled surface has neither distinguished
polygonal pattern nor distinguished enumeration of the squares.
\end{NNConvention}

It  is  immediate  to  see  that the square-tiled surface $S(N)$ from
Figure~\ref{fig:palindrome}  belongs  to  the stratum $\cH(2g-2)$ for
$N=2g-1$,  and  to  the stratum $\cH(g-1,g-1)$ for $N=2g$. Clearly, a
central  symmetry  of  the  pattern  from Figure~\ref{fig:palindrome}
extends  to  an  involution  of  the surface. This involution has two
fixed  points at the waist curve of the cylinder. It also has a fixed
point  at  the middle of each of the $N$ horizontal sides of squares.
The  involution  fixes the single zero when $N=2g-1$ and interchanges
the  two  zeroes  when  $N=2g$. Thus, the involution has $2g+2$ fixed
points,  so  it  is  a hyperelliptic involution. We conclude that the
surface  $S(N)$  belongs  to  the  hyperelliptic  connected component
$\cH^{hyp}(2g-2)$,  when $N=2g-1$, and to the hyperelliptic connected
component $\cH^{hyp}(g-1,g-1)$, when $N=2g$.

\begin{Proposition}
\label{pr:palindrome}
The Hodge bundle $H^1_{\R{}}$ over the arithmetic Teichm\"uller curve
of  the  surface  $S(N)$  from Figure~\ref{fig:palindrome} decomposes
into   a   direct   sum   of  $\SL$-invariant,  Hodge  star-invariant
two-dimensional symplectic subbundles
$$
H^1_{\R{}}\simeq \bigoplus_{k\le N/2}W_+(k)
$$
The  Lyapunov  spectrum  $\operatorname{{\Lambda}Spec}$  of the Hodge
bundle  $H^1_{\R{}}$  over the corresponding arithmetic Teichm\"uller
curve is
$$
\operatorname{{\Lambda}Spec}=
\begin{cases}
\cfrac{1}{N}, \cfrac{3}{N}, \cfrac{5}{N}, \dots, \cfrac{N}{N}
       & \text{ when } N=2g-1\\&\\
\cfrac{2}{N}, \cfrac{4}{N}, \cfrac{6}{N}, \dots, \cfrac{N}{N} & \text{ when } N=2g
\end{cases}
$$
\end{Proposition}
\begin{proof}
We  show  in Lemma~\ref{lm:quotient:odd} that for odd $N$ the surface
$S(N)$  in Figure~\ref{fig:palindrome} corresponds to a quotient of a
cyclic  cover $M_{2N}(2N-1,1,N,N)$ over the involution $\tau$ defined
by                equation~\eqref{eq:involution}.               Using
Theorem~\ref{th:main:theorem}   and   Corollary~\ref{cr:alorithm}  we
compute  in Lemma~\ref{lm:spectrum:odd} the Lyapunov spectrum for the
Teichm\"uller     curve    of    $M_{2N}(2N-1,1,N,N)$    and    apply
Proposition~\ref{pr:symmetry}.

Similarly, we show in Lemma~\ref{lm:quotient:even} that for even
$N$, the surface $S(N)$ in Figure~\ref{fig:palindrome} corresponds to
a  quotient  of a square-tiled cyclic cover $M_{N}(N-1,1,N-1,1)$ over
the    involution    $\tau$    as   in~\eqref{eq:involution}.   Using
Theorem~\ref{th:main:theorem}   and   Corollary~\ref{cr:alorithm}  we
compute  the Lyapunov spectrum for the arithmetic Teichm\"uller curve
of $M_{N}(N-1,1,N-1,1)$ and apply Proposition~\ref{pr:symmetry}.
\end{proof}

\begin{Lemma}
\label{lm:quotient:odd}
For   odd  $N$  the  surface  $S(N)$  in  Figure~\ref{fig:palindrome}
corresponds   to  the  quotient  of  the  square-tiled  cyclic  cover
$M_{2N}(2N-1,1,N,N)$     over     the     involution     $\tau$    as
in~\eqref{eq:involution}.
\end{Lemma}
\begin{proof}
Note  that  for  any  integer  $a<2N$  such  that  $\gcd(2N,a)=1$ the
square-tiled cyclic covers
$$
M_{2N}(2N-1,1,N,N) \simeq M_{2N}(2N-a,a,N,N)
$$
are  isomorphic. To simplify combinatorics it would be easier to work
with $M_{2N}(N+2,N-2,N,N)\simeq M_{2N}(2N-1,1,N,N)$.

\begin{figure}[htb]
   %
\includegraphics{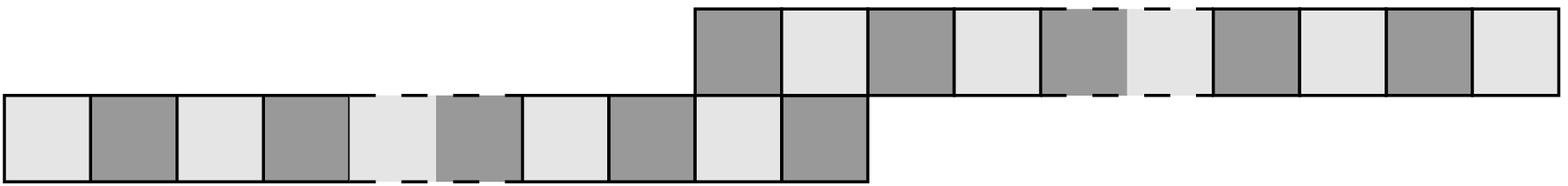}
\includegraphics{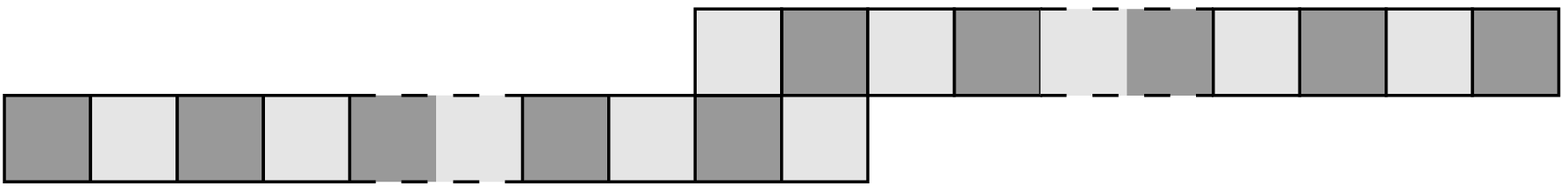}
\begin{picture}(0,0)(178,16.5) 
\begin{picture}(0,0)(0,0)
\put(10,-25){\tiny 0}
\put(30,-25){\tiny 1}
\put(50,-25){\tiny 2}
\put(70,-25){\tiny 3}
\put(96,-25){$\cdots$}
\put(123,-25){\tiny 2N-4}
\put(143,-25){\tiny 2N-3}
\put(163,-25){\tiny 2N-2}
\put(184,-25){\tiny 2N-1}
\put(5,-10){\tiny N+1}
\put(29,-10){\tiny N}
\put(45,-10){\tiny N+3}
\put(65,-10){\tiny N+2}
\put(125,-10){\tiny N-3}
\put(145,-10){\tiny N-4}
\end{picture}
\begin{picture}(0,0)(-155,-20)
\put(5,-25){\tiny N-1}
\put(25,-25){\tiny N-2}
\put(45,-25){\tiny N-3}
\put(65,-25){\tiny N-4}
\put(96,-25){$\cdots$}
\put(124,-25){\tiny N+3}
\put(144,-25){\tiny N+2}
\put(164,-25){\tiny N+1}
\put(189,-25){\tiny N}
\put(2,-10){\tiny 2N-2}
\put(22,-10){\tiny 2N-1}
\put(42,-10){\tiny 2N-4}
\put(62,-10){\tiny 2N-3}
\put(130,-10){\tiny 2}
\put(150,-10){\tiny 3}
\put(170,-10){\tiny 0}
\put(190,-10){\tiny 1}
\end{picture}
\end{picture}
   %
   %
\begin{picture}(0,0)(182,87) 
\begin{picture}(0,0)(0,0)
\put(5,-25){\tiny N-1}
\put(25,-25){\tiny N-2}
\put(45,-25){\tiny N-3}
\put(65,-25){\tiny N-4}
\put(96,-25){$\cdots$}
\put(124,-25){\tiny N+3}
\put(144,-25){\tiny N+2}
\put(164,-25){\tiny N+1}
\put(189,-25){\tiny N}
\put(4,-10){\tiny 2N-2}
\put(24,-10){\tiny 2N-1}
\put(44,-10){\tiny 2N-4}
\put(64,-10){\tiny 2N-3}
\put(130,-10){\tiny 2}
\put(150,-10){\tiny 3}
\end{picture}
\begin{picture}(0,0)(-155,-20)
\put(10,-25){\tiny 0}
\put(30,-25){\tiny 1}
\put(50,-25){\tiny 2}
\put(70,-25){\tiny 3}
\put(96,-25){$\cdots$}
\put(123,-25){\tiny 2N-4}
\put(143,-25){\tiny 2N-3}
\put(163,-25){\tiny 2N-2}
\put(184,-25){\tiny 2N-1}
\put(5,-10){\tiny N+1}
\put(29,-10){\tiny N}
\put(45,-10){\tiny N+3}
\put(65,-10){\tiny N+2}
\put(125,-10){\tiny N-3}
\put(145,-10){\tiny N-4}
\put(165,-10){\tiny N-1}
\put(185,-10){\tiny N-2}
\end{picture}
\end{picture}
\vspace{125bp}
\caption{
\label{fig:odd}
Square-tiled $M_{2N}(N+2,N-2,N,N)$, where $N$ is odd. A square
atop a white one is always black and vice versa.
}
\end{figure}

Suppose  that the powers $N+2,N-2,N,N$ are distributed at the corners
of a ``square pillow'' representing our flat $\CP$ as follows:

\includegraphics{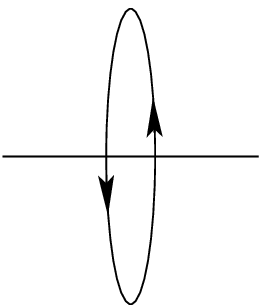}
\includegraphics{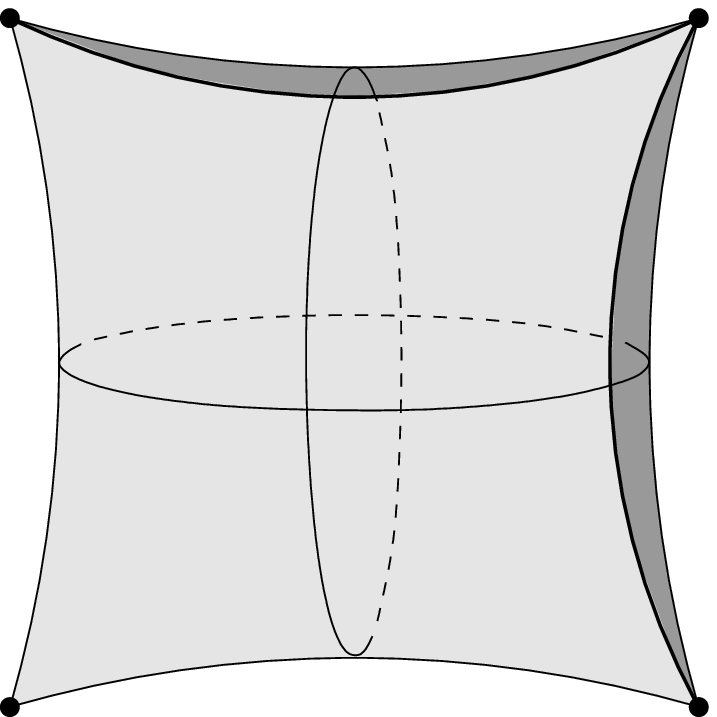}
\begin{picture}(0,0)(5,7)
\put(125,-3){$N+2$}
\put(125,-33){$N-2$}
\put(190,-3){$N$}
\put(190,-33){$N$}
\end{picture}
\vspace{50bp}

Let  us  color  the  ``visible''  square  of the ``square pillow'' in
white,  and the complementary square in black. It is easy to see that
Figure~\ref{fig:odd}    represents    two    unfoldings   (preserving
enumeration  of  the  squares)  of the same square-tiled cyclic cover
$M_{2N}(N+2,N-2,N,N)$,  where we enumerate separately white and black
squares  by  numbers  from  $0$  to  $2N-1$.  Moreover,  the  natural
generator  of  the group of deck transformations maps a square number
$k$ to a square of the same color having number $(k+1)\!\mod 2N$.

\begin{figure}[htb]
   %
\special{
psfile=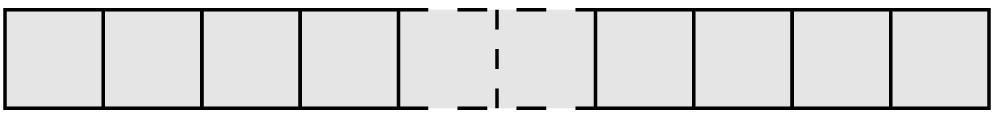
hscale=70
vscale=70
hoffset=82 
voffset=-35 
}
\includegraphics{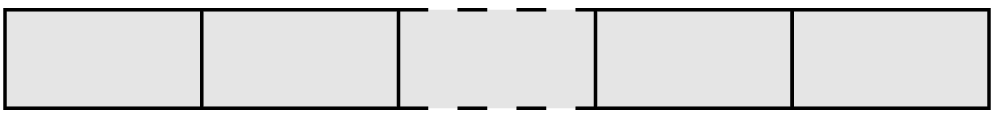}
\begin{picture}(0,0)(0,0)
\begin{picture}(0,0)(97,2)
\put(10,-25){\tiny 0}
\put(30,-25){\tiny 1}
\put(50,-25){\tiny 2}
\put(70,-25){\tiny 3}
\put(84,-25){$\cdots$}
\put(106,-25){$\cdots$}
\put(123,-25){\tiny 2N-4}
\put(143,-25){\tiny 2N-3}
\put(163,-25){\tiny 2N-2}
\put(184,-25){\tiny 2N-1}
\put(3,-10){\tiny 2N-2}
\put(23,-10){\tiny 2N-1}
\put(44,-10){\tiny 2N-4}
\put(65,-10){\tiny 2N-3}
\put(130,-10){\tiny 2}
\put(150,-10){\tiny 3}
\put(170,-10){\tiny 0}
\put(190,-10){\tiny 1}
\end{picture}
\begin{picture}(0,0)(97,52)
\put(17,-25){\tiny A}
\put(57,-25){\tiny B}
\put(90,-25){$\cdots$}
\put(133,-25){\tiny Y}
\put(173,-25){\tiny Z}
\put(17,-10){\tiny Z}
\put(57,-10){\tiny Y}
\put(133,-10){\tiny B}
\put(173,-10){\tiny A}
\end{picture}
\end{picture}
\vspace{90bp}
\caption{
\label{fig:odd:quotient}
Quotient  of  a  square-tiled cyclic cover $M_{2N}(N+2,N-2,N,N)$ over
the involution $\tau$.}
\end{figure}

In  enumeration of Figure~\ref{fig:odd} the involution $\tau$ defined
in equation~\eqref{eq:involution} acts as
$$
\tau(k_{\mathit{white}})=(N-1-k)_{\mathit{black}} \ ,
$$
which  corresponds  to  a  superposition  of  the two patterns of our
square-tiled  surface indicated at Figure~\ref{fig:odd}. Passing to a
quotient  over $\tau$ we obtain a square-tiled surface represented at
Figure~\ref{fig:odd:quotient}. Note, that the top horizontal sides of
each pair of adjacent squares having numbers $2k$ and $2k+1$ is glued
to  bottom  horizontal  sides of a pair of consecutive squares number
$2N-2-2k$   and   $2N-1-2k$.  Hence,  we  can  tile  the  surface  at
Figure~\ref{fig:odd:quotient}  with  rectangles  of size $2\times 1$.
Rescaling  the  resulting  surface by a contraction in the horizontal
direction  and  by  an  expansion  in  a  vertical  direction,  which
transforms  rectangles  of size $2\times 1$ into squares, we obtain a
square-tiled surface $S(N)$ as at Figure~\ref{fig:palindrome}.
\end{proof}

\begin{Lemma}
\label{lm:spectrum:odd}
The  spectrum  of  nonnegative Lyapunov exponents of the Hodge bundle
$H^1_{\R{}}$  over  the  Teichm\"uller curve of a square-tiled cyclic
cover $M_{2N}(2N-1,1,N,N)$ has the following form for odd $N>1$:
$$
\operatorname{{\Lambda}Spec}=
\left\{
\cfrac{1}{N},\ \cfrac{1}{N},\ \cfrac{3}{N},\ \cfrac{3}{N},\  \dots\ ,\cfrac{N-2}{N},\ \cfrac{N-2}{N},\ 1
\right\}
$$
\end{Lemma}

\begin{proof}
Since $\gcd(2N,2N-1)=\gcd(2N,1)=1$, we conclude that
$$
t_1(k)>0\quad \text{ and }\quad
t_2(k)>0\quad\text{for } k=1,\dots,N-1\ ,
$$
where
$$
t_1(k)=\left\{\frac{2N-1}{2N}\cdot k\right\} \text{ and }
t_2(k)=\left\{\frac{1}{2N}\cdot k\right\}\ ,
$$
see~\eqref{eq:tik}. Since
$$
\frac{2N-1}{2N}\cdot k+
\frac{1}{2N}\cdot k=k\in\Z
$$
we obtain
$$
t_1(k)+t_2(k)=1\quad\text{for }k=1,\dots,N-1\ .
$$
Note that
$$
t_3(k)=t_4(k)=\left\{\frac{N}{2N}\cdot k\right\}=
\left\{\frac{k}{2}\right\}\ .
$$
Hence,    by    definition~\eqref{eq:tk}    of    $t(k)$    and    by
formula~\eqref{eq:dim:duality} we get
$$
\dim_{\C{}} V^{1,0}(k)=\dim_{\C{}} V^{1,0}(N-k)=
\begin{cases}
1&\text{ when $k$ is odd}\\
0&\text{ when $k$ is even}\\
\end{cases}
$$
Finally, our computation shows that
\begin{multline*}
2d(k)=2
\min(t_1(k),1-t_1(k),\dots,t_4(k),1-t_4(k))=\\
=\ \frac{k}{N}\quad\text{ for }k=1,3,\dots,2j+1,\dots,N\,.
\end{multline*}
The     statement     of     the     Lemma     follows    now    from
Corollary~\ref{cr:alorithm}.
\end{proof}

\begin{Lemma}
\label{lm:quotient:even}
For any even $N$ the surface $S(N)$ as in Figure~\ref{fig:palindrome}
corresponds   to   a   quotient   of   a  square-tiled  cyclic  cover
$M_{N}(N-1,1,N-1,1)$     over     an     involution     $\tau$     as
in~\eqref{eq:involution}.
\end{Lemma}

\begin{figure}[htb]
\centering
\includegraphics{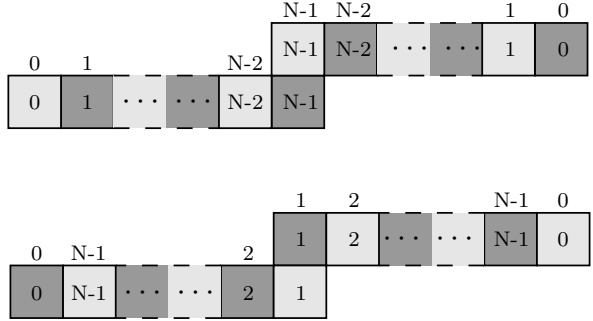}
\includegraphics{M_6_5511.eps}
\begin{picture}(0,0)(11,5) 
\begin{picture}(0,0)(97,2)
\put(10,-25){\tiny 0}
\put(30,-25){\tiny 1}
\put(45,-25){$\cdots$}
\put(65,-25){$\cdots$}
\put(85,-25){\tiny N-2}
\put(106,-25){\tiny N-1}
\end{picture}
\begin{picture}(0,0)(101,-13)
\put(10,-25){\tiny 0}
\put(30,-25){\tiny 1}
\put(85,-25){\tiny N-2}
\end{picture}
\end{picture}
\begin{picture}(0,0)(-89,-15) 
\begin{picture}(0,0)(97,2)
\put(2,-25){\tiny N-1}
\put(22,-25){\tiny N-2}
\put(43,-25){$\cdots$}
\put(63,-25){$\cdots$}
\put(86,-25){\tiny 1}
\put(106,-25){\tiny 0}
\put(2,-10){\tiny N-1}
\put(22,-10){\tiny N-2}
\put(86,-10){\tiny 1}
\put(106,-10){\tiny 0}
\end{picture}
\end{picture}
   %
\begin{picture}(0,0)(18,77) 
\begin{picture}(0,0)(97,2)
\put(10,-25){\tiny 0}
\put(25,-25){\tiny N-1}
\put(45,-25){$\cdots$}
\put(65,-25){$\cdots$}
\put(90,-25){\tiny 2}
\put(110,-25){\tiny 1}
\end{picture}
\begin{picture}(0,0)(101,-13)
\put(10,-25){\tiny 0}
\put(25,-25){\tiny N-1}
\put(90,-25){\tiny 2}
\end{picture}
\end{picture}
\begin{picture}(0,0)(-78,57) 
\begin{picture}(0,0)(97,2)
\put(10,-25){\tiny 1}
\put(30,-25){\tiny 2}
\put(43,-25){$\cdots$}
\put(63,-25){$\cdots$}
\put(85,-25){\tiny N-1}
\put(109,-25){\tiny 0}
\put(10,-10){\tiny 1}
\put(30,-10){\tiny 2}
\put(85,-10){\tiny N-1}
\put(109,-10){\tiny 0}
\end{picture}
\end{picture}
\vspace{115bp}
\caption{
\label{fig:even}
Square-tiled $M_{N}(N-1,1,N-1,1)$ for even $N$. A square atop a white
one is always black and vice versa. }
\end{figure}

\begin{proof}
The  proof  is  analogous  to the one of Lemma~\ref{lm:quotient:odd}.
Suppose  that the powers $N-1,1,N-1,1$ are distributed at the corners
of a ``square pillow'' representing our flat $\CP$ as follows:

\includegraphics{axes_hor.eps}
\includegraphics{pillow_white.eps}
\begin{picture}(0,0)(5,7)
\put(125,-3){$N-1$}
\put(151,-33){$1$}
\put(190,-3){$N-1$}
\put(190,-33){$1$}
\end{picture}
\vspace{50bp}

Figure~\ref{fig:even}   represents   two  unfoldings  (preserving  an
enumeration   of   squares)   of   the   square-tiled   cyclic  cover
$M_{N}(N-1,1,N-1,1)$,  where  we enumerate separately white and black
squares  by  numbers  from $0$ to $N-1$. The natural generator of the
group of deck transformations maps a square number $k$ to a square of
the same color having number $(k+1)\mod N$.

\begin{figure}[htb]
   %
\includegraphics{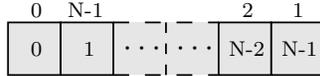}
\begin{picture}(0,0)(-34,-15) 
\begin{picture}(0,0)(97,2)
\put(10,-25){\tiny 0}
\put(30,-25){\tiny 1}
\put(45,-25){$\cdots$}
\put(65,-25){$\cdots$}
\put(85,-25){\tiny N-2}
\put(105,-25){\tiny N-1}
\put(10,-10){\tiny 0}
\put(24,-10){\tiny N-1}
\put(90,-10){\tiny 2}
\put(109,-10){\tiny 1}
\end{picture}
\end{picture}
\vspace{25bp} 
\caption{
\label{fig:even:quotient}
Quotient  of a square-tiled cyclic cover $M_{N}(N-1,1,N-1,1)$ over an
involution.}
\end{figure}

In enumeration of squares as in Figure~\ref{fig:even}, the involution
$\tau$ defined in equation~\eqref{eq:involution} acts as
$$
\tau(k_{\mathit{white}})=(N-k)_{\mathit{black}}
$$
which  corresponds  to  a  superposition  of  the two patterns of our
square-tiled surface indicated at Figure~\ref{fig:even}. Passing to a
quotient  over $\tau$ we obtain a square-tiled surface represented at
Figure~\ref{fig:even:quotient}.   Clearly,   it  coincides  with  the
square-tiled surface $S(N)$ as at Figure~\ref{fig:palindrome}.
\end{proof}

\begin{Lemma}
\label{lm:spectrum:even}
The  spectrum  of  nonnegative Lyapunov exponents of the Hodge bundle
$H^1_{\R{}}$  over  the  Teichm\"uller curve of a square-tiled cyclic
cover $M_{N}(N-1,1,N-1,1)$ has the following form for even $N>2$:
$$
\operatorname{{\Lambda}Spec}=
\left\{
\cfrac{2}{N},\ \cfrac{2}{N},\ \cfrac{4}{N},\ \cfrac{4}{N},\  \dots\ ,\cfrac{N-2}{N},\ \cfrac{N-2}{N},\ 1
\right\}
$$
\end{Lemma}
\begin{proof}
The    proof    is    completely    analogous   to   the   proof   of
Lemma~\ref{lm:spectrum:odd} and is left to the reader.
\end{proof}

Proposition~\ref{pr:palindrome} is proved.

\begin{Remark}
\label{rm:S2N:to:SN}
Note  that  the  surface  $S(2N)$  as  in Figure~\ref{fig:palindrome}
admits    an    automorphism    $\sigma:    S\to    S$    such   that
$\sigma^\ast\omega=\omega$.         In         enumeration         of
Figure~\ref{fig:even:quotient} the corresponding permutation $\sigma$
is defined as
$$
\sigma(k)=k+N\!\mod 2N \ .
$$
One  easily checks that $S(2N)/\sigma=S(N)$. For even $N$ the induced
double  cover  $S(2N)\to S(N)$ has two ramification points at the two
zeroes  of  $S(2N)$.  For  odd  $N$  the  cover  $S(2N)\to  S(N)$  is
unramified.

This  cover  extends  to an involution of the corresponding universal
curve, and, thus, induces a map of the Teichm\"uller curve of $S(2N)$
to   the  Teichm\"uller  curve  of  $S(N)$.  The  latter  map  is  an
isomorphism  for  even  $N$  and  a  double  cover  for  odd $N$, see
section~\ref{ss:orbits} below.
\end{Remark}

\begin{Remark}
Note  that  the  quotient  $\tau$  of  a  square-tiled  cyclic  cover
$M_{2N}(2N-1,1,2N-1,1)$  over $S(2N)$ \textit{does not} extends to an
automorphism of the universal curve. In particular, the Teichm\"uller
curve   of  the  square-tiled  $M_{2N}(2N-1,1,2N-1,1)$  is  a  double
\textit{quotient}  of the one of $S(2N)$, see section~\ref{ss:orbits}
below.  There is no contradiction with Proposition~\ref{pr:symmetry}:
if    we    \textit{give    names}    to    the    four   zeroes   of
$M_{2N}(2N-1,1,2N-1,1)$ breaking part of the symmetry (as it was done
in  Proposition~\ref{pr:symmetry}),  the  Teichm\"uller  curve of the
square-tiled  $M_{2N}(2N-1,1,2N-1,1)$  will  become isomorphic to the
one of the $S(2N)$.

A  square-tiled  cyclic  cover $M_{2N}(2N-1,1,2N-1,1)$ admits another
involution     $\tau_2$     preserving     the     flat    structure,
$\tau_2\omega=\omega$.  It is induced from a unique involution of the
underlying  $\CP$  which  interchanges  the points in each of the two
pairs  $P_1,P_3$ and $P_2,P_4$ (the points in each pair correspond to
the   \textit{same}   power  $2N-1$  or  $1$  correspondingly).  This
automorphism commutes with the deck transformations,
$
\tau_2 T=T \tau_2\ .
$
The  quotient  of a square-tiled cyclic cover $M_{2N}(2N-1,1,2N-1,1)$
over  $\tau_2$  is  isomorphic to a cyclic cover $M_{2N}(2N-1,1,N,N)$
tiled with rectangles $1\times\frac{1}{2}$.
\end{Remark}

\begin{Remark}
A  square-tiled  cyclic  cover  $M_{2N}(2N-1,1,N,N)$  is  obviously a
double cover over $M_N(N-1,1,N,N)$. Clearly, a quadratic differential
$p^\ast q_0$ induced from $q_0$ on $\CP$ by the projection
$$
p:M_N(N-1,1,N,N)\to\CP
$$
has  two  zeroes  of  degrees $N-2$ and $2N$ simple poles. Hence, the
cyclic cover $M_N(N-1,1,N,N)$ is a Riemann sphere, which implies that
$M_{2N}(2N-1,1,N,N)$  is hyperelliptic Riemann surface. Note that the
square-tiled cyclic cover $M_{2N}(2N-1,1,N,N)$ belongs to the stratum
$\cH(N-1,N-1)$.  However,  the  hyperelliptic involution of this flat
surface   fixes   the   zeroes,  so  the  square-tiled  cyclic  cover
$M_{2N}(2N-1,1,N,N)$  belongs  to  the component $\cH^{odd}(N-1,N-1)$
and   not   to   the   hyperelliptic  component  (see  Proposition  7
in~\cite{Kontsevich:Zorich}).
\end{Remark}

\subsection{Orbits of several distinguished square-tiled surfaces}
\label{ss:orbits}

For  the  sake  of  completeness  we  describe  the  $\SLZ$-orbits of
square-tiled  surfaces  discussed  above.  Clearly,  these  data also
describes the structure of the corresponding arithmetic Teichm\"uller
curve.

All   surfaces   considered   below   are   hyperelliptic,  so  their
$\SLZ$-orbits  coincide  with $\PSLZ$-orbits. We denote generators of
$\SLZ$ by
$
T=\begin{pmatrix}1&1\\0&1\end{pmatrix}
$
and
$
S=\begin{pmatrix}0&1\\-1&0\end{pmatrix}\ .
$

The  Veech  groups  of  the ``\textit{stairs}'' square-tiled surfaces
which  appear  in  the  orbits  of  surfaces  $S(N)$  were  found  by
G.~Schmith\"usen      in~\cite{Schmithusen};      the     generalized
``\textit{stairs}''  square-tiled  surfaces are studied by M.~Schmoll
in~\cite{Schmoll:stairs}.

\newpage

\includegraphics{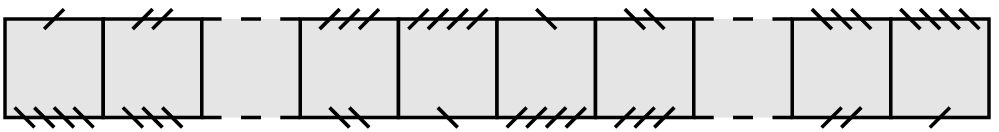}

\includegraphics{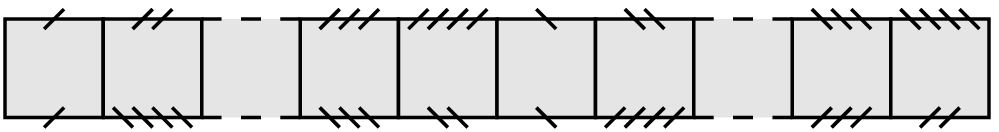}

\includegraphics{palindrome10_1.eps}

\includegraphics{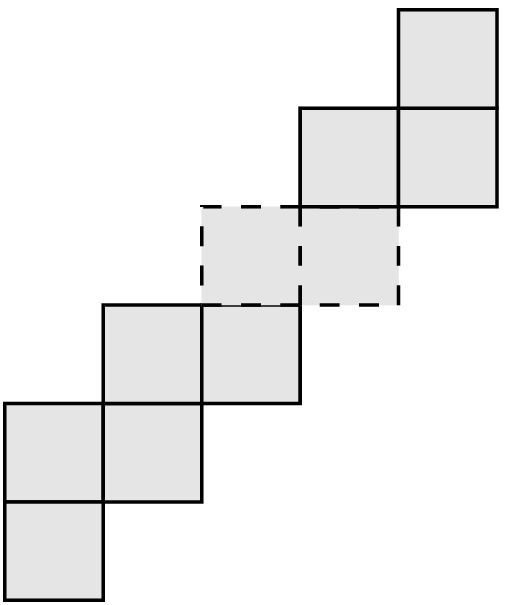}

\includegraphics{stair3.eps}

\includegraphics{palindrome10_2.eps}

\includegraphics{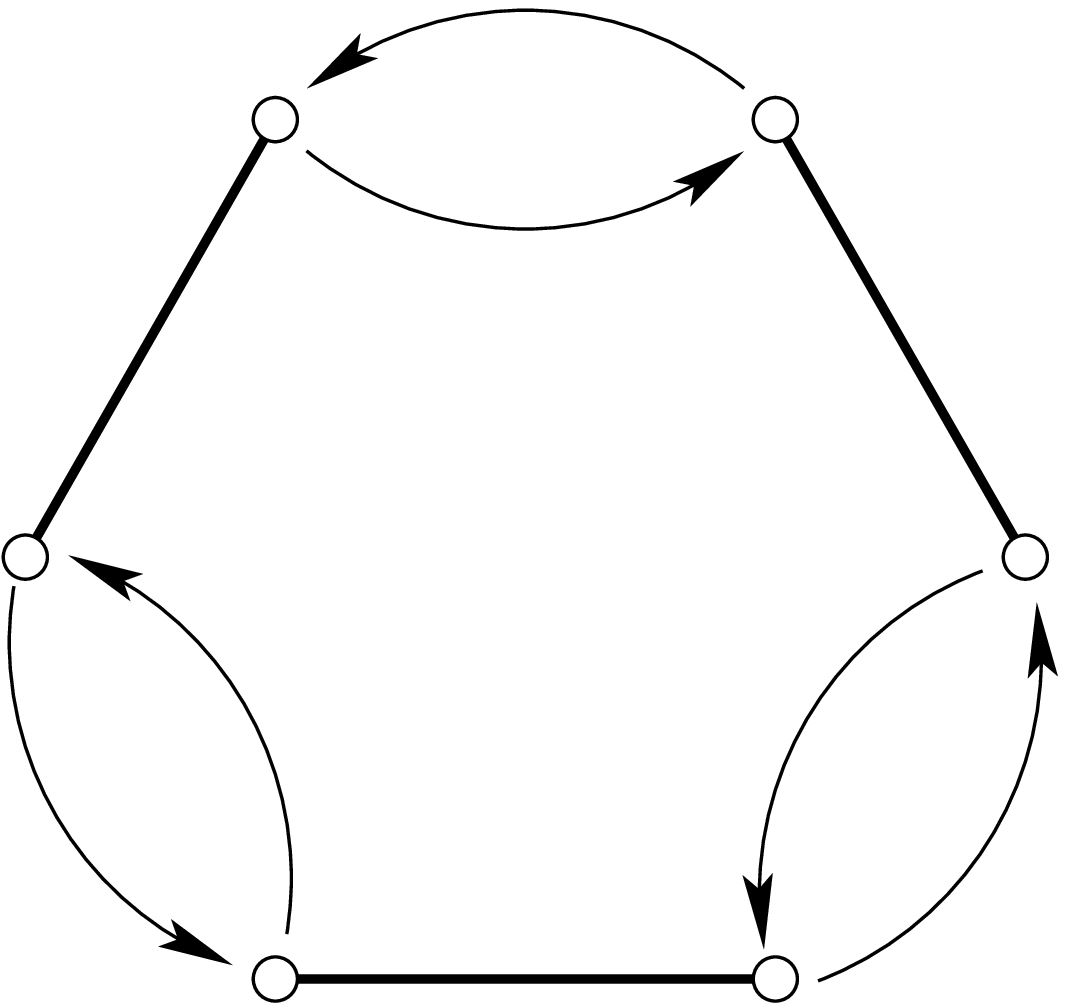}

\begin{picture}(0,0)(-200,0)
\put(-39,-40){$T$}
\put(-65,-87){$T$}
\put(-10,-87){$T$}
\put(-39,-110){$S$}
\put(-70,-55){$S$}
\put(-3,-55){$S$}
\put(-60,-35){\scriptsize $1$}
\put(-57.5,-32){\circle{10}}
\put(-10,-35){\scriptsize $2$}
\put(-7.5,-32){\circle{10}}
\put(-80,-70){\scriptsize $6$}
\put(-77.5,-67){\circle{10}}
\put(8,-70){\scriptsize $3$}
\put(10.5,-67){\circle{10}}
\put(-59,-108){\scriptsize $5$}
\put(-56.5,-105){\circle{10}}
\put(-15,-108){\scriptsize $4$}
\put(-12.5,-105){\circle{10}}
\put(-94,-21){\scriptsize $1$}
\put(27,-21){\scriptsize $2$}
\put(-195,-63){\scriptsize $6$}
\put(125,-63){\scriptsize $3$}
\put(-110,-130){\scriptsize $5$}
\put(43,-130){\scriptsize $4$}
\begin{picture}(0,0)(-2.5,-2.5)
\put(-94,-21){\circle{10}}
\put(27,-21){\circle{10}}
\put(-195,-63){\circle{10}}
\put(125,-63){\circle{10}}
\put(-110,-130){\circle{10}}
\put(43,-130){\circle{10}}
\end{picture}
\end{picture}

\vspace*{150bp}

The picture above (correspondingly below) represents the $\SLZ$-orbit
of  the  surface  $S(N)$ from Figure~\ref{fig:palindrome} when $N$ is
even   (correspondingly  odd).  The  bottom  picture  represents  the
$\SLZ$-orbit of the square-tiled $M_N(N-1,1,N-1,1)$.

\includegraphics{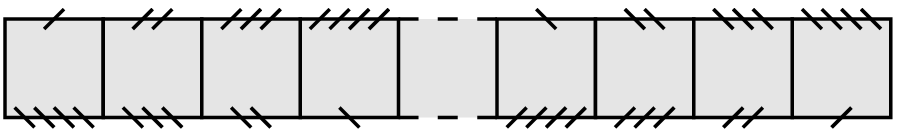}

\includegraphics{palindrome9.eps}

\includegraphics{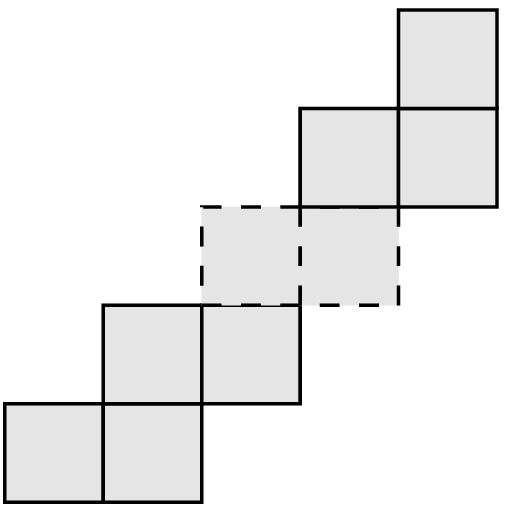}

\begin{picture}(0,0)(-260,0)
\put(-155,-56){\scriptsize $1$}
\put(-105,-56){\scriptsize $2$}
\put(-10,-54) {\scriptsize  $3$}
\begin{picture}(0,0)(-2.5,-2.5)
\put(-155,-56){\circle{10}}
\put(-105,-56){\circle{10}}
\put(-10,-54) {\circle{10}}
\end{picture}
\end{picture}

\includegraphics{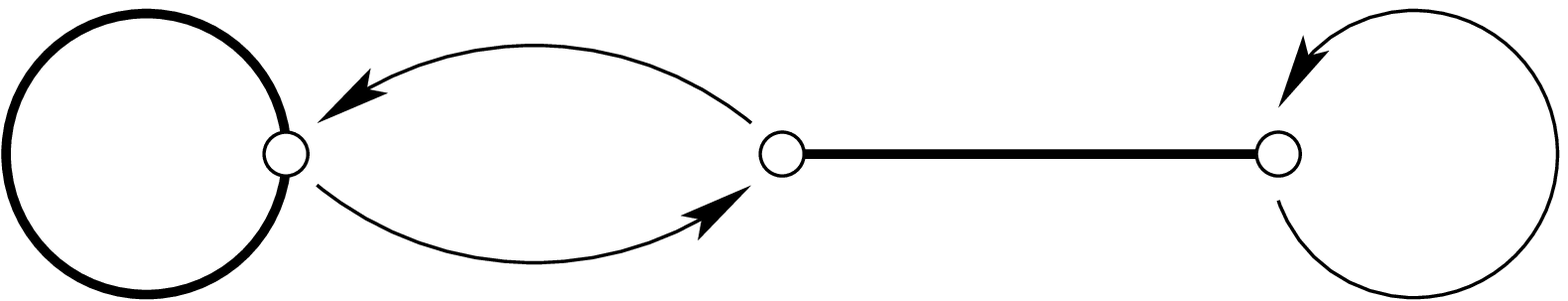}
\begin{picture}(0,0)(-260,80)
\put(-156,-68){$T$}
\put(-116,-62){$S$}
\put(-66,-68){$T$}
\put(-7,-62){$S$}
\put(-135,-74){\scriptsize $1$}
\put(-98,-74){\scriptsize $2$}
\put(-31,-68){\scriptsize  $3$}
\begin{picture}(0,0)(-2.5,-2.5)
\put(-135,-74){\circle{10}}
\put(-98,-74){\circle{10}}
\put(-31,-68){\circle{10}}
\end{picture}
\end{picture}


\vspace*{180bp}

\includegraphics{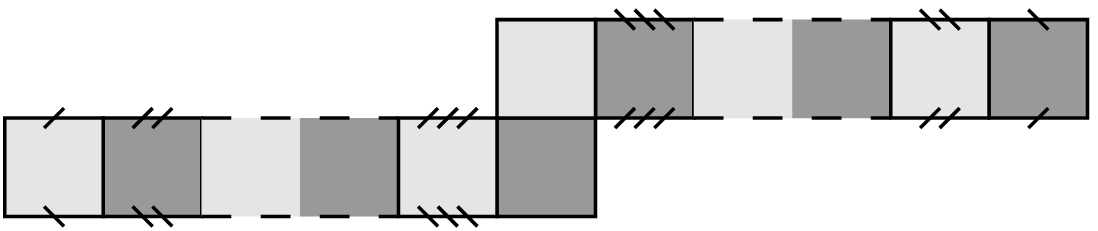}

\includegraphics{M_6_5511_identif.eps}

\includegraphics{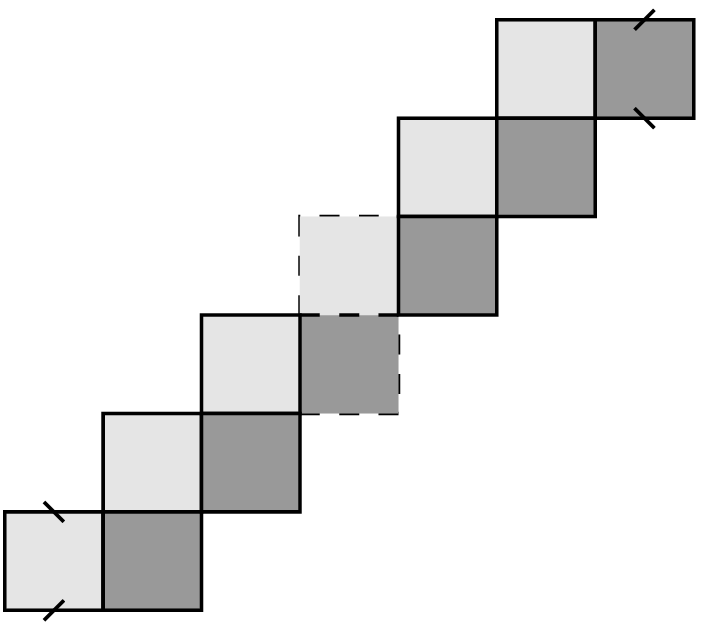}

\begin{picture}(0,0)(-260,0)
\put(-138,-52.5){\scriptsize $1$}
\put(-102,-53){\scriptsize $2$}
\put(-10,-53) {\scriptsize  $3$}
\begin{picture}(0,0)(-2.5,-2.5)
\put(-138,-52.5){\circle{10}}
\put(-102,-53){\circle{10}}
\put(-10,-53) {\circle{10}}
\end{picture}
\end{picture}

\newpage
\section{Homological dimension of the square-tiled $M_N(N-1,1,N-1,1)$}

The  following  characteristic of a square-tiled surface is important
for  applications, see~\cite{Forni:new}. Given a square-tiled surface
$S_0$,  consider  a collection of cycles $c_1,\dots,c_j$ representing
the  waist  curves  of  its  maximal  horizontal  cylinders,  and let
$d(S_0)$  be  the  dimension of the linear span of $c_1,\dots,c_j$ in
$H_1(S_0,\R{})$.  The \textit{homological dimension} of an arithmetic
Teichm\"uller  curve  corresponding  to a square-tiled surface $S_0$ is
defined  as  the  maximum  of $d(S)$ over $S$ in the $\PSLZ$-orbit of
$S_0$.

Answering  the  question  of  G.~Forni  we  show that the homological
dimension  of the arithmetic Teichm\"uller curve corresponding to the
square-tiled  cyclic cover $M_N(N-1,1,N-1,1)$ for even $N$ is maximal
possible:  it  is equal to the genus $g$ of the corresponding Riemann
surface.

\begin{figure}[hbt]
\includegraphics{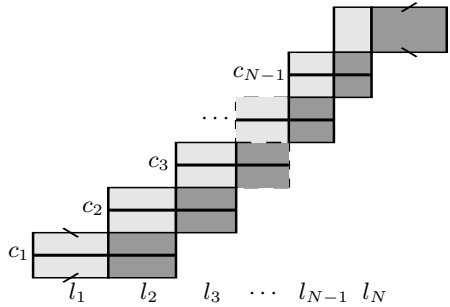}
\begin{picture}(0,0)(-95,-5) 
\put(-180,-95){\scriptsize $c_1$}
\put(-152,-78){\scriptsize $c_2$}
\put(-127,-61) {\scriptsize  $c_3$}
\put(-107,-44.5) {\scriptsize  $\cdots$}
\put(-95,-27) {\scriptsize  $c_{N-1}$}
\put(-157,-111){\scriptsize $l_1$}
\put(-130,-111){\scriptsize $l_2$}
\put(-106,-111) {\scriptsize  $l_3$}
\put(-89,-111) {\scriptsize  $\cdots$}
\put(-70,-111) {\scriptsize  $l_{N-1}$}
\put(-46,-111) {\scriptsize  $l_N$}
\end{picture}
\vspace{100pt} 
\caption{
\label{fig:4:Giovanni}
A   horizontal   deformation   of   the   square-tiled  cyclic  cover
$M_N(N-1,1,N-1,1)$
}
\end{figure}

The square-tiled cyclic cover $M_N(N-1,1,N-1,1)$ was discussed in the
previous  section. It has four conical singularities with cone angles
$N\pi$.  The  corresponding  Abelian  differential has four zeroes of
degrees  $N/2-1$.  Hence  the  underlying  Riemann  surface has genus
$g=N-1$.

The   $\PSLZ$-orbit   of   the   square-tiled  $M_N(N-1,1,N-1,1)$  is
represented    at    the    bottom    picture    at    the   end   of
section~\ref{ss:orbits}.  Consider  a deformation of surface $3$ from
this  orbit  as in Figure~\ref{fig:4:Giovanni}. By $l_1,\dots,l_N$ we
denote  the  lengths  of  the  horizontal  sides of the corresponding
rectangles.

The  integrals of the holomorphic form $\omega$ representing the flat
structure over the cycles $c_1,\dots,c_{N-1}$ are equal to $l_1+l_2$,
$l_2+l_3$,     $\dots$,    $l_{N-1}+l_N$    correspondingly.    Since
$l_1,\dots,l_N$  are  arbitrary  positive  real numbers, we can chose
them in such a way that the integrals over $c_1,\dots,c_{N-1}$ become
rationally independent. Hence, the integer cycles $c_1,\dots,c_{N-1}$
are  linear  independent over rationals and, thus, linear independent
over  reals.  Since  $g=N-1$  they  span a $g$-dimensional Lagrangian
subspace.

\newpage

\section{Cyclic cover $M_{30}(3,5,9,13)$}
\label{a:table}

The table below evaluates the quantities discussed in this paper in a
concrete case of the square-tiled cyclic cover $M_{30}(3,5,9,13)$.

$$
\begin{array}{|c||c|c|c|c||c||c|c|c||c|}
\hline &\multicolumn{4}{|c||}{}&&&&&\\
[-\halfbls]
k & \multicolumn{4}{|c||}{t_i(k)=\left\{\cfrac{a_i}{N}k\right\}}&
\!t(k)\! &\!    \dim V^{1,0}\! &\!    \dim V^{0,1}\! &\!    \dim V\! &    \lambda\\
[-\halfbls]&\multicolumn{4}{|c||}{}&&&&&\\
\hline
1&1/10&1/6&3/10&13/30&1&0&2&2&-\\
\hline
2&1/5&1/3&3/5&13/15&2&1&1&2&4/15\\
\hline
3&3/10&1/2&9/10&3/10&2&1&1&2&1/5\\
\hline
4&2/5&2/3&1/5&11/15&2&1&1&2&2/5\\
\hline
5&1/2&5/6&1/2&1/6&2&1&1&2&1/3\\
\hline
6&3/5&0&4/5&3/5&2&1&0&1&0\\
\hline
7&7/10&1/6&1/10&1/30&1&0&2&2&-\\
\hline
8&4/5&1/3&2/5&7/15&2&1&1&2&2/5\\
\hline
9&9/10&1/2&7/10&9/10&3&2&0&2&0;\ 0\\
\hline
10&0&2/3&0&1/3&1&0&0&0&-\\
\hline
11&1/10&5/6&3/10&23/30&2&1&1&2&1/5\\
\hline
12&1/5&0&3/5&1/5&1&0&1&1&-\\
\hline
13&3/10&1/6&9/10&19/30&2&1&1&2&1/5\\
\hline
14&2/5&1/3&1/5&1/15&1&0&2&2&-\\
\hline
15&1/2&1/2&1/2&1/2&2&1&1&2&1\\
\hline
16&3/5&2/3&4/5&14/15&3&2&0&2&0;\ 0\\
\hline
17&7/10&5/6&1/10&11/30&2&1&1&2&1/5\\
\hline
18&4/5&0&2/5&4/5&2&1&0&1&0\\
\hline
19&9/10&1/6&7/10&7/30&2&1&1&2&1/5\\
\hline
20&0&1/3&0&2/3&1&0&0&0&-\\
\hline
21&1/10&1/2&3/10&1/10&1&0&2&2&-\\
\hline
22&1/5&2/3&3/5&8/15&2&1&1&2&2/5\\
\hline
23&3/10&5/6&9/10&29/30&3&2&0&2&0;\ 0\\
\hline
24&2/5&0&1/5&2/5&1&0&1&1&-\\
\hline
25&1/2&1/6&1/2&5/6&2&1&1&2&1/3\\
\hline
26&3/5&1/3&4/5&4/15&2&1&1&2&2/5\\
\hline
27&7/10&1/2&1/10&7/10&2&1&1&2&1/5\\
\hline
28&4/5&2/3&2/5&2/15&2&1&1&2&4/15\\
\hline
29&9/10&5/6&7/10&17/30&3&2&0&2&0;\ 0\\
\hline
\end{array}
$$
\bigskip

Here   $V^{1,0},  V^{0,1},  V$  denote  $V^{1,0}(N-k),  V^{0,1}(N-k),
V(N-k)$ correspondingly.

Lyapunov  exponents  in the right column of the table are constructed
following  the algorithm from Corollary~\ref{cr:alorithm}. The cyclic
cover $M_{30}(3,\!5,\!9,\!13)$ has genus $g=25$. The table also illustrates
Lemma~\ref{lm:dim:Vk:equals:two}:  depending on how many of $t_i(k)$,
where  $i=1,2,3,4$,  are  null for a given $k$, the eigenspace $V(k)$
might have dimension $0,1$ or $2$.

The    resulting   spectrum   $\{\lambda_1,\dots,\lambda_{25}\}$   is
presented below:
$$
\left\{1,
\ \left(\cfrac{2}{5}\right)^4,
\ \left(\cfrac{1}{3}\right)^2,
\ \left(\cfrac{4}{15}\right)^2,
\ \left(\cfrac{1}{5}\right)^6,
\ 0^{10}
\right\}
$$
where powers denote multiplicities.

\subsection*{Acknowledgments}
We  thank  Giovanni  Forni,  Carlos  Matheus, and Martin M\"oller for
important  discussions.  We  thank  Alex  Wright  and  the  anonymous
referees for careful reading the manuscript and indicating to us some
typos.  Appendix~\ref{s:symmetries}  is  motivated  by  a question of
J.-C.~Yoccoz, for which we would like to thank him.

The authors are grateful to HIM, IHES, MPIM and University of Chicago
for hospitality while preparation of this paper.

\end{document}